\documentclass{article}

\usepackage{amsmath}   
\usepackage{amssymb}
\usepackage{amsthm}

\newtheorem{Th}{Theorem}[section] 
\newtheorem{Prop}{Proposition}[section]   

\newtheorem{Lem}{Lemma}[section]   
\newtheorem{Coro}{Corollary}[section]   
  
\newtheorem{Rem}{Remark}[section]

\newcommand{\R}{\mathbb{R}}
\newcommand{\Rr}{\mathbb{R}_r}
\newcommand{\Z}{\mathbb{Z}}

\newcommand{\x}{\langle x\rangle}
\newcommand{\rr}{\langle r\rangle}
\newcommand{\A}{{\mathcal A}}
\newcommand{\AH}{{\mathbb A}}
\newcommand{\D}{{\mathcal D}}
\newcommand{\F}{{\mathcal F}}
\newcommand{\U}{{\mathcal U}}
\newcommand{\V}{{\mathcal V}}

\newcommand{\Diff}{{\rm Diff}}
\newcommand{\id}{{\rm id}}
\newcommand{\supp}{\mathop{\rm supp}\nolimits}

\begin{document}

\title{Asymptotics in shallow water waves}   
 
\author{Robert McOwen \& Peter Topalov} 
  
\maketitle

\begin{abstract}  
In this paper we consider the initial value problem for a family of shallow water equations on the line
$\R$ with various asymptotic conditions at infinity. In particular we construct solutions with
prescribed asymptotic expansion as $x\to\pm\infty$ and prove their invariance with respect to the 
solution map.
\end{abstract}

\section{Introduction}\label{sec:introduction}
In this paper we study the initial value problem for a family of shallow water equations on the
line $\R$  that contains the Camassa-Holm equation (CH),
\begin{equation}\label{eq:ch}
\left\{
\begin{array}{l}
m_t+u m_x+2 m u_x=0,\;\;m:=u-u_{xx}\\
u|_{t=0}=u_0
\end{array}
\right.
\end{equation}
with initial data $u_0$ that has an asymptotic expansion of order $N\in \Z_{\ge 0}$,
\begin{equation}\label{eq:asymptotics}
u_0(x)=\left\{
\begin{array}{l}
c_0^++\frac{c_1^+}{x}+...+\frac{c_N^+}{x^N}+o(|x|^{-N})\;\;\mbox{as}\;\;x\to+\infty\\
c_0^-+\frac{c_1^-}{x}+...+\frac{c_N^-}{x^N}+o(|x|^{-N})\;\;\mbox{as}\;\;x\to-\infty,
\end{array}
\right.
\end{equation}
where $c_k^\pm$ are real constants for $0\le k\le N$. (Note that the asymptotic expansion of $u_0$ as
$x\to+\infty$ is {\em not} necessarily the same as the one when $x\to-\infty$.)
We are interested to study whether \eqref{eq:ch} possesses a solution with
the prescribed initial data and if the corresponding solution map (provided it exists) preserves the
asymptotic expansion \eqref{eq:asymptotics} in the sense described below.
We postpone for the moment the questions related to the regularity of the initial data $u_0$ and
the solutions.

Note that similar questions were studied for the Korteweg-de Vries equation (KdV) in a series of papers
\cite{BS1,BS2,BS3} as well as in \cite{KPST} for the modified KdV equation. In the present paper we
develop a different approach to this problem. The new approach is influenced by the seminal paper
of Arnold \cite{Arn} (cf. also the related papers \cite{EM,OK,Mis1,Con2}) and is based on the introduction
of a {\em group of asymptotic diffeomorphisms} on the line. The group of asymptotic diffeomorphisms is an
infinite dimensional topological group modeled on a Banach space of functions on the line with prescribed
asymptotic expansions at infinity. This provides a general framework for studying the asymptotics of solutions
for a relatively large class of nonlinear equations. As a model we consider equation \eqref{eq:ch} 
(as well as its generalization \eqref{eq:b-family} below). The approach is {\em not}
restricted to equations with one spacial dimension (cf. \cite{McOwenTopalov1}). 

Since the time when equation \eqref{eq:ch} was derived (in \cite{FF} using algebraic principles and in
\cite{CH} as a model for shallow water waves), it has been attracting a lot of
attention. One of the reasons is that  it is completely integrable and, unlike the classical
Korteweg-de Vries equation, admits  non-smooth solitary traveling waves $u(x,t)=c e^{-|x-ct|}$ called
{\em peakons}.\footnote{Note that the $x$-derivative of $u(t,x)$ has a jump of magnitude $2c$ at $x=ct$.}
Equation \eqref{eq:ch} is a particular case of a $1$-parameter family of shallow water
equations (\cite{HS1}),
\begin{equation}\label{eq:b-family}
\left\{
\begin{array}{l}
m_t+u m_x+b m u_x=0,\;\;m=u-u_{xx}\\
u|_{t=0}=u_0
\end{array}
\right.
\end{equation}
where $b$ is a real parameter. When $b=2$ we get \eqref{eq:ch}; the case $b=3$
is the Degasperis-Procesi equation (DP). Both equations are known to be completely integrable
and to admit solitary traveling waves. It is worth noting that, although the cases when $b\ne2, 3$ in 
\eqref{eq:b-family} are not known to be completely integrable, they still admit solitary
traveling waves (\cite{HS2}). 

The main result of this paper states that, under some technical assumptions on the remainder term 
$o(|x|^{-N})$ in \eqref{eq:asymptotics}, for any $b\in\R$ and for any $m\ge 3$ and $N\ge 0$ there
exists $T>0$ and a solution $u\in C^0([0,T],H^m_{loc}(\R))\cap C^1([0,T],H^{m-1}_{loc}(\R))$ of 
\eqref{eq:b-family} so that for any $t\in[0,T]$, $u(t)$ has the asymptotic expansion 
\eqref{eq:asymptotics} with coefficients $c_k^\pm$ for $0\leq k\leq N$ that may depend on
$t\in[0,T]$ -- see Theorem \ref{th:main1} and  Theorem \ref{th:main2} below for the precise
statement of the result. Here $H^m_{loc}(\R)$ denotes the space of measurable real valued functions
on $\R$ whose weak derivatives up to order $m$ are locally square integrable.

\vspace{.3cm}

\noindent {\em Analytic set-up:} 
In order to perform analysis on functions satisfying \eqref{eq:asymptotics} we assume that
the remainder term $o(|x|^{-N})$ lies in a suitable weighted Sobolev space. The main idea is to 
incorporate the asymptotic terms in \eqref{eq:asymptotics} into a new functional space that we call an
{\em asymptotic space}. It turns out that the asymptotic spaces defined this way are Banach spaces
that enjoy the Banach algebra property in the sense that the (pointwise) product of functions is
continuous. Another important feature of these spaces is that they admit a natural (graded) Lie algebra
structure that corresponds to the group structure of a special class of diffeomorphisms of $\R$
(see below). In this paper we restrict our attention to the following asymptotic spaces:

\vspace{.3cm}

\noindent $(1)$ {\em The asymptotic space $\A_{n,N}^m$.}
Take $m\ge 1$ and $N\ge 0$.
In order to define the asymptotic space $\A_{n,N}^m(\R)$, we first define the weighted Sobolev space 
\begin{equation}\label{eq:W_N^m}
W_N^m(\R):=\{f\in H^m_{loc}(\R)\,|\,\x^Nf, \x^{N+1}f',..., \x^{N+m}f^{(m)}\in L^2(\R)\},
\end{equation}
supplied with the norm
\[
\|f\|_{W^m_N}:=\Big(\sum_{j=0}^m\int_{\R}|\x^{N+j}f^{(j)}(x)|^2\,dx\Big)^{1/2},
\]
where $\x:=\sqrt{1+x^2}$ and $f^{(j)}=\partial_x^jf$ denotes the $j$-th weak derivative of $f$.
Note that $f\in W^m_N(\R)$ implies (Lemma \ref{lem:decay})
\begin{equation}\label{eq:small_o}
f(x)=o(|x|^{-N})\;\;\mbox{as}\;\;x\to\pm\infty\,.
\end{equation}
One easily sees that \eqref{eq:asymptotics} is equivalent to
\begin{equation}\label{eq:asymptotics'}
u_0(x)=\sum_{k=0}^N\Big(a_k\frac{1}{\x^k}+b_k\frac{x}{\x^{k+1}}\Big)+o(\x^{-N})
\end{equation}
where the constants $a_k, b_k\in\R$, $0\le k\le m$, are uniquely determined from $c_k^\pm$, $0\le k\le N$.
In view of this we define the asymptotic space
\[
\A^m_N(\R):=\Big\{u=\sum_{k=0}^N\Big(a_k\frac{1}{\x^k}+b_k\frac{x}{\x^{k+1}}\Big)+
f\,\Big|\,f\in W^m_N(\R)\Big\},
\]
\[
\|u\|_{\A^m_N}:=\sum_{k=0}^N(|a_k|+|b_k|)+\|f\|_{W_N^m}\,.
\]
By \eqref{eq:small_o}, the elements of $\A^m_N(\R)$ satisfy the asymptotic expansion
\eqref{eq:asymptotics'} (and equivalently, \eqref{eq:asymptotics}).
More generally, for $0\le n\le N$, define
\[
\A_{n,N}^m(\R):=\Big\{u=\sum_{k=n}^N\Big(a_k\frac{1}{\x^k}+b_k\frac{x}{\x^{k+1}}\Big)+
f\,\Big|\,f\in W^m_N(\R)\Big\},
\]
\[
\|u\|_{\A_{n,N}^m}:=\sum_{k=n}^N(|a_k|+|b_k|)+\|f\|_{W_N^m}\,.
\]
We set,
\begin{equation}\label{eq:convention1}
\A_{n,N}^m(\R):=W_N^m(\R)\,\,\,\mbox{for}\,\,\,n\ge N+1\,.
\end{equation}

\vspace{.3cm}

\noindent For any $R>0$ denote by $B_{\A^m_{n,N}}(R)$ the ball of radius $R$ in $\A^m_{n,N}$ centered at
the origin. The following theorem is proved in Section \ref{sec:A}.
\begin{Th}\label{th:main1}
For any $b\in\R$, $m\ge 3$, $N\ge 0$, $n\ge 1$ and $R>0$, there exists $T>0$ such that for any 
$u_0\in B_{\A^m_{n,N}}(R)$ there exists a unique solution 
$u\in C^0([0,T],\A_{n,N}^m(\R))\cap C^1([0,T],A^{m-1}_{n,N}(\R))$ of \eqref{eq:b-family}
that depends continuously on the initial data $u_0\in B_{\A^m_{n,N}}(R)$ in the sense that the data-to-solution map,
\[
u_0\mapsto u,\,\,\,B_{\A^m_{n,N}}(R)\to C^0([0,T],\A_{n,N}^m(\R))\cap C^1([0,T],\A^{m-1}_{n,N}(\R))\,,
\]
is continuous. Moreover, the coefficients $a_k$, $b_k$, $n\le k\le\min\{2n,N\}$, in the asymptotic expansion
of the solution $u(t)$ are independent of $t$.
\end{Th}

\vspace{.3cm}

\noindent $(2)$ {\em The asymptotic space $\AH_{n,N}^m$.}
Although the choice of the remainder space $W_N^m$ is natural as it mimics the properties
of the asymptotic terms $1/\x^k$ and $x/\x^{k+1}$ with respect to the differentiation in the $x$-direction, 
this choice is not unique: one can choose other Banach spaces of functions that satisfy the {\em remainder condition} 
\eqref{eq:small_o}. For any $m\ge 1$ and $N\ge 0$ consider the weighted Sobolev space,
\begin{equation}\label{eq:H_N^m}
H_N^m(\R):=\{f\in H^m_{loc}(\R)\,|\,\x^Nf, \x^Nf',..., \x^Nf^{(m)}\in L^2(\R)\},
\end{equation}
\[
\|f\|_{H^m_N}:=\Big(\sum_{j=0}^m\int_{\R}|\x^N f^{(j)}(x)|^2\,dx\Big)^{1/2}\,.
\]
One has the following continuous inclusions:
$
W_N^m(\R)\subseteq H_N^m(\R)\subseteq H^m(\R),
$
and, for $1\le l\le N$,
$
H_N^l(\R)\subseteq W_{N-l}^l(\R)\,.
$
Moreover, the elements of $H^m_N(\R)$ satisfy the remainder condition \eqref{eq:small_o}
(see Lemma \ref{lem:decay'}).
Hence, for $0\le n\le N$, we can define in a similar way as above the asymptotic space,
\begin{equation}\label{eq:AH_N^m}
\AH^m_{n,N}(\R):=\Big\{u=\sum_{k=n}^N\Big(a_k\frac{1}{\x^k}+b_k\frac{x}{\x^{k+1}}\Big)+
f\,\Big|\,f\in H^m_N(\R)\Big\},
\end{equation}
\[
\|u\|_{\AH^m_{n,N}}:=\sum_{k=n}^N(|a_k|+|b_k|)+\|f\|_{H_N^m}\,.
\]
We set,
\begin{equation}\label{eq:convention2}
\AH_{n,N}^m(\R):=H_N^m(\R)\,\,\,\mbox{for}\,\,\,n\ge N+1\,.
\end{equation}

\vspace{.3cm}

\noindent For any $R>0$ denote by $B_{\AH^m_{n,N}}(R)$ the ball of radius $R$ in $\AH^m_{n,N}$ centered at
the origin. The following theorem is proved in Section \ref{sec:AH}.
\begin{Th}\label{th:main2}
For any $b\in\R$, $m\ge 3$, $N\ge 0$, $n\ge 0$ and $R>0$, there exists $T>0$ such that for any 
$u_0\in B_{\AH^m_{n,N}}(R)$, there exists a unique solution 
$u\in C^0([0,T],\AH_{n,N}^m(\R))\cap C^1([0,T],\AH^{m-1}_{n,N}(\R))$ of \eqref{eq:b-family}
that depends continuously on the initial data $u_0\in B_{\AH^m_{n,N}}(R)$ in the sense that the data-to-solution map,
\[
u_0\mapsto u,\;\;\;B_{\AH^m_{n,N}}(R)\to C^0([0,T],\AH_{n,N}^m(\R))\cap C^1([0,T],\AH^{m-1}_{n,N}(\R))\,,
\]
is continuous. Moreover, the coefficients $a_k$, $b_k$, $n\le k\le\min\{2n,N\}$, in the asymptotic expansion of
the solution $u(t)$ are independent of $t$.
\end{Th}
\begin{Rem}\label{rem:constant_term}
Note  that, unlike in Theorem \ref{th:main1},  the case $n=0$ is {\em not} excluded in Theorem \ref{th:main2}.
Hence, the class of solutions obtained in Theorem \ref{th:main2} contains, in particular, solutions $u(t)$ such that
for any $t\in[0,T]$,  $u(t,x)\to c_0^\pm$ as $x\to\pm\infty$. The constants $c_0^\pm$ are {\em not} necessarily zero,
or equal to each other. This implies that such solutions are {\em not} necessarily summable or square integrable.
\end{Rem}
\begin{Rem}
In view of Theorem \ref{th:main1} and Theorem \ref{th:main2} the coefficients $c_k^\pm$, 
$n\le k\le n_*:=\max\{2n,N\}$ in the asymptotic expansion of the solution $u$ are {\em conservation laws} of
equation \eqref{eq:b-family} for any $b\in\R$. 
Hence we get $2(n_*-n+1)$ functionally independent integrals of motion.
This is in contrast to the KdV and the modified KdV equation where only the leading coefficients $c_n^\pm$
are preserved (\cite{KPST}). 
\end{Rem}
\begin{Rem}
By taking $n=N+1$  in Theorem \ref{th:main1} and Theorem \ref{th:main2} we obtain, in 
view of conventions \eqref{eq:convention1}  and  \eqref{eq:convention2}, that \eqref{eq:b-family} is
well-posed in the remainder spaces $W^m_N$ and $H^m_N$ ($m\ge 3$, $N\ge 0$) respectively.
The case of the space $W^3_1$ was considered by Constantin in \cite{Con2} while the case of $H^3=H^3_0$ was 
treated by Constantin-Escher in \cite{CE2}.
\end{Rem}
\begin{Rem}
As the goal of this work is to study the spacial asymptotics of the solutions we do not attempt to lower the
regularity exponent $m$ in our spaces. Note however that the method allows for analogs of Theorem \ref{th:main1} and 
Theorem \ref{th:main2} to be proved with $m>3/2$. We refer to \cite{CE4,DKT,Mis2} where low-regularity results are
proved for the CH equation. In order to keep the paper as non-technical as possible, 
we also do not consider the case of the asymptotic $L^p$-spaces $\A_{n,N}^{m,p}$ and $\AH_{n,N}^{m,p}$, $p\ge 1$
(see \cite{McOwenTopalov1} for the definitions and the main properties of these spaces). 
Since we always assume $p=2$, we simplify the notation by omitting reference to $p$.
\end{Rem}
\begin{Rem}
The definition of the asymptotic group in Section \ref{sec:AD} can be extended by allowing linear terms 
$c^\pm_{-1} x$, $c^\pm_{-1}\ne -1$ in the asymptotic expansion \eqref{eq:asymptotics} and
the corresponding asymptotic spaces. Note that similarly to KdV, the CH equation \eqref{eq:ch} possesses
the unbounded solution $u(t,x):=x/[3(t-1)]$ that, at $t=1$, blows-up at all points $x\in\R$ except $x=0$.
\end{Rem}

\noindent{\em Organization of the paper:} The paper is organized as follows. In Section \ref{sec:AD} we define
the group of asymptotic diffeomorphisms on the line and discuss its main properties. At the end of the section
we formulate Proposition \ref{prop:lagrangian_coordinates} that establishes a relation between the solutions of 
\eqref{eq:b-family} and the solutions of a dynamical system formulated in terms of the asymptotic group.
Section \ref{sec:A} is devoted to the proof of Theorem \ref{th:main1}. Theorem \ref{th:main2} is proved in
Section \ref{sec:AH}.  Proposition \ref{prop:lagrangian_coordinates} is proved in Appendix A.
In Appendix B we study the properties of the operator $1-\partial_x^2$ acting on weighted spaces
and, for the convenience of the reader, prove in much detail several auxiliary results needed in the main 
body of the paper.

\section{Groups of asymptotic diffeomorphisms}\label{sec:AD}
Denote by $\Diff_+^1(\R)$ the group of orientation preserving $C^1$-diffeomorphisms on the line $\R$.
For any $m\ge 2$, $N\ge 0$, and $n\ge 0$ define
\[
\A\D_{n,N}^m:=\{\varphi\in\Diff_+^1(\R)\,|\,\varphi(x)=x+u(x), u\in\A_{n,N}^m\}
\]
and
\[
\AH\D_{n,N}^m:=\{\varphi\in\Diff_+^1(\R)\,|\,\varphi(x)=x+u(x), u\in\AH_{n,N}^m\}\,.
\]
The topology on $\A\D_{n,N}^m$ (resp.\  $\AH\D_{n,N}^m$) is inherited in a natural way from the Banach structure of
$\A_{n,N}^m$ (resp.\ $\AH_{n,N}^m$). The following theorem follows from the analysis in \cite{McOwenTopalov1}. 

\begin{Th}\label{th:asymptotic_group}
$1)$ For any $m\ge 2$, $N\ge 0$, $n\ge 0$, the composition of maps
\begin{equation}\label{eq:composition}
\circ : \A^m_{n,N}\times\A\D^m_{n,N}\to\A^m_{n,N},\ (v,\varphi)\mapsto v\circ\varphi \quad\hbox{is continuous}
\end{equation}
and
\[
\circ : \A^{m+1}_{n,N}\times\A\D^m_{n,N}\to\A^m_{n,N},\ (v,\varphi)\mapsto v\circ\varphi \quad\hbox{is $C^1$-smooth.}
\]
\noindent $2)$ For any $m\ge 2$, $N\ge 0$, $n\ge 0$, mapping $\varphi$ to its inverse $\varphi^{-1}$
\[
{\rm Inv} : \A\D^{m+1}_{n,N}\to\A\D^{m+1}_{n,N},\ \varphi\mapsto\varphi^{-1} \quad\hbox{is continuous}
\]
and
\[
{\rm Inv} : \A\D^{m+1}_{n,N}\to\A\D^m_{n,N},\ \varphi\mapsto\varphi^{-1} \quad\hbox{is $C^1$-smooth.}
\]

\noindent $3)$ The same statements in $1)$ and $2)$ are true if $\A$ is replaced by $\AH$.
\end{Th}

\begin{Coro}
For $m\geq 3$, $N\geq 0$, $n\geq 0$, both $\A\D_{n,N}^m$ and $\AH\D_{n,N}^m$ are topological groups with
respect to composition.
\end{Coro}

\noindent
We call $\A\D_{n,N}^m$ and $\AH\D_{n,N}^m$ {\it groups of asymptotic diffeomorphisms}. Note that 
\[
\A\D_{n,N}^m \subseteq \AH\D_{n,N}^m \subseteq \Diff_+^1(\R).
\]
Let us also mention another group of diffeomorphisms that was studied in \cite{IKT}:
\[
\D^m(\R):=\{\varphi\in\Diff_+^1(\R)\,|\,\varphi(x)=x+u(x), u\in H^m\}\,.
\]
Observe that the inclusion $\AH\D_{n,N}^m\subseteq\D^m(\R)$ holds for $n\geq 1$ but 
$\D^m(\R)\not\subseteq\AH\D_{0,0}^m$; in fact, by convention \eqref{eq:convention1},
$\D^m(\R)=\AH\D_{1,0}^m$.

\begin{Rem} For simplicity we have omitted $\R$ in the notation for $\A\D_{n,N}^m$ and $\AH\D_{n,N}^m$. 
Henceforth in this paper we shall also omit $\R$ in the notation for other function spaces.
\end{Rem}

The significance of these groups of asymptotic diffeomorphisms stems from their use for introducing
Lagrangian coordinates. This is summarized in the following two results (in which we use $\dot{\ }$ to denote
the $t$-derivative).

\begin{Lem}\label{lem:ode}
Let $m\ge 3$, $N\ge 0$, $n\ge 0$, and $T>0$. Assume that $u\in C^0([0,T],\A^m_{n,N})$.
Then there exists a unique 
$\varphi\in C^1([0,T],\A D^m_{n,N})$ so that
\begin{eqnarray}\label{eq:ode}
\dot\varphi=u\circ\varphi,\,\,\,\varphi|_{t=0}=\id\,.
\end{eqnarray}
The same statement is true if $\A$ is replaced by $\AH$.
\end{Lem}
\begin{Rem}
This is an important lemma. It shows in particular that although the asymptotic group $\A\D^m_{n,N}$ is not a Lie group
in classical sense (as the composition \eqref{eq:composition} is {\em not} $C^\infty$-smooth\,\footnote{This map is {\em not} even
Lipschitz continuous.}) one can still define the {\em Lie-group exponential map},
\[
\exp_{\A\D^m_{n,N}} : \A^m_{n,N}\to\A\D^m_{n,N},\,\,\,u\mapsto\varphi(1),
\]
where $\varphi\in C^1([0,\infty),\A D^m_{n,N})$ is the solution of \eqref{eq:ode} and $u\in\A^m_{n,N}$
is independent of $t$. In this sense, the asymptotic space $\A^m_{n,N}$ can be thought as the ``Lie algebra''
of $\A\D^m_{n,N}$.
\end{Rem}
The following Proposition establishes a relation between the solutions of \eqref{eq:b-family} and the solutions of a 
dynamical system formulated in terms of the asymptotic group.
\begin{Prop}\label{prop:lagrangian_coordinates}
Assume that $m\ge 3$, $N\ge 0$, $n\ge 1$, and $u_0\in\A^m_{n,N}$.
Then there exists a {\em bijective} correspondence between solutions of equation \eqref{eq:b-family} 
in $C^0([0,T],\A^m_{n,N})\cap C^1([0,T],\A^{m-1}_{n,N})$ and solutions of
\begin{equation}\label{eq:lagrangian_coordinates}
\left\{
\begin{array}{l}
(\dot\varphi,\dot v)=
\Big(v,R_\varphi\circ(1-\partial_x^2)^{-1}\circ\beta_b\circ R_{\varphi^{-1}}(v)\Big)\\
(\varphi,v)|_{t=0}=(\id,u_0)
\end{array}
\right.
\end{equation}
in $C^1\big([0,T],\A D^m_{n,N}\times\A^m_{n,N}\big)$, where $\beta_b(u)=-b\,u u_x-(3-b)u_x u_{xx}$
and $R_\varphi(u)=u\circ\varphi$. 
More specifically, if $(\varphi,v)\in C^1\big([0,T],\A D^m_{n,N}\times\A^m_{n,N}\big)$
is a solution of \eqref{eq:lagrangian_coordinates} then 
\[
u=v\circ\varphi^{-1}
\]
is a solution of \eqref{eq:b-family} in $C^0([0,T],\A^m_{n,N})\cap C^1([0,T],\A^{m-1}_{n,N})$.
The same statements are true if $\A$ is replaced by $\AH$; 
we then allow $n\geq 0$.
\end{Prop}

\noindent
Lemma \ref{lem:ode} and Proposition \ref{prop:lagrangian_coordinates} are proved in Appendix A. 
They will be used in Section \ref{sec:A} to prove Theorem \ref{th:main1} and in Section \ref{sec:AH}
to prove Theorem \ref{th:main2}.

\section{Analysis on the asymptotic space $\A_{n,N}^m$}\label{sec:A}
In this section we prove Theorem \ref{th:main1}. The proof is based on
Proposition \ref{prop:lagrangian_coordinates} that establishes a correspondence between solutions
of \eqref{eq:b-family} and the dynamical system \eqref{eq:lagrangian_coordinates} on
$\A D^m_{n,N}\times\A^m_{n,N}$. Our first goal is to prove that the right hand side of the first equation in 
\eqref{eq:lagrangian_coordinates} defines a $C^\infty$-smooth vector field on an open neighborhood
of $(\id,0)$ in $\A D^m_{n,N}\times\A^m_{n,N}$.

\subsection{Mapping Properties of $(1-\partial_x^2)^{-1}$ on $\A_{n,N}^m$}\label{sec:mapping_properties}
Let $m\ge 1$, $N\ge 0$, and $n\ge 0$. In order to study the mapping properties of the operator
$\Lambda:=1-\partial_x^2$ and its inverse on $\A_N^m$ we introduce the modified weighted Sobolev space
\begin{equation}
{\widetilde W}_N^{m+2}:=
\{f\in H^{m+2}_{loc}\,|\,f\in W_N^m, \x^{N+m}f^{(m+1)}, \x^{N+m}f^{(m+2)}\in L^2\},
\end{equation}
\[
\|f\|_{{\widetilde W}_N^{m+2}}:=
\|f\|_{W_N^m}+\|\x^{N+m}f^{(m+1)}\|_{L^2}+\|\x^{N+m}f^{(m+2)}\|_{L^2}\,.
\]
Accordingly, for $n\ge 0$ we define the modified asymptotic space
\[
\tilde\A_{n,N}^{m+2}:=\Big\{u=\sum_{k=n}^N\Big(a_k\frac{1}{\x^k}+b_k\frac{x}{\x^{k+1}}\Big)+
f\,\Big|\,f\in{\widetilde W}^{m+2}_N(\R)\Big\},
\]
\[
\|u\|_{\tilde\A_{n,N}^{m+2}}:=\sum_{k=n}^N(|a_k|+|b_k|)+\|f\|_{{\widetilde W}_N^{m+2}}\,.
\]
It follows directly from the definition of ${\widetilde W}_N^{m+2}$ and $\tilde\A_{n,N}^{m+2}$ that
\begin{eqnarray}
{\widetilde W}_N^{m+2}\subseteq W_N^m,&\,\,\tilde\A_{n,N}^{m+2}\subseteq\A_{n,N}^m
\;\;(N\ge 0)\label{eq:inclusion1}\\
{\widetilde W}_N^{m+2}\subseteq W_{N-1}^{m+1},&\,\,\tilde\A_{n,N}^{m+2}\subseteq\A_{n,N-1}^{m+1}
\;\;(N\ge 1)\label{eq:inclusion2}\\
{\widetilde W}_N^{m+2}\subseteq W_{N-2}^{m+2},&\,\,\tilde\A_{n,N}^{m+2}\subseteq\A_{n,N-2}^{m+2}
\;\;(N\ge 2)\label{eq:inclusion3}
\end{eqnarray}
where the inclusions are continuous. There is a natural splitting
\begin{equation}\label{eq:splitting}
\A_{n,N}^m=A_{n,N}\oplus  W_N^m\;\;\;\mbox{and}\;\;\;
\tilde\A_{n,N}^{m+2}=A_{n,N}\oplus{\widetilde W}_N^{m+2}
\end{equation}
where the finite-dimensional space
\[
A_{n,N}:=\Big\{\sum_{k=n}^N\Big(a_k\frac{1}{\x^k}+b_k\frac{x}{\x^{k+1}}\Big)\,\Big|\,a_k,b_k\in\R\Big\}
\]
is supplied with the norm $\sum_{k=n}^N(|a_k|+|b_k|)$.
First, we prove the following Lemma.
\begin{Lem}\label{lem:L_on_W} 
For any $m\ge 1$, $N\ge 0$, and $n\ge 0$, the mapping,
\begin{equation}\label{eq:L_on_W}
\Lambda : {\widetilde W}_N^{m+2}\to W_N^m,
\end{equation}
is a linear isomorphism.
\end{Lem}
\begin{proof}
Take $f\in{\widetilde W}_N^{m+2}$. For any $0\le j\le m$,
\begin{equation}\label{eq:equality1}
\x^{N+j}[(1-\partial_x^2) f^{(j)}] = \x^{N+j} f^{(j)}-\x^{N+j} f^{(j+2)}\,.
\end{equation}
By the definition of ${\widetilde W}_N^{m+2}$, for any $0\le j\le m$,
\begin{equation}
\x^{N+j} f^{(j)}\in L^2,
\end{equation}
and for any $0\le j\le m-2$, with $m\ge 2$,
\begin{equation}
|\x^{N+j} f^{(j+2)}|\le \x^{N+j+2} |f^{(j+2)}|\in L^2\,.
\end{equation}
In addition, for any $m\ge 1$,
\begin{equation}
|\x^{N+m-1} f^{(m+1)}|\le \x^{N+m} |f^{(m+1)}|\in L^2
\end{equation}
and
\begin{equation}\label{eq:equality1'}
\x^{N+m} f^{(m+2)}\in L^2\,.
\end{equation}
Combining together \eqref{eq:equality1}-\eqref{eq:equality1'} we see that $\Lambda(f)\in W_N^m$ and
\eqref{eq:L_on_W} is continuous.
As $\Lambda : H^{m+2}\to H^m$ is a linear isomorphism and as $ {\widetilde W}_N^{m+2}\subseteq H^{m+2}$,
the mapping \eqref{eq:L_on_W} is injective.

Next, we prove that \eqref{eq:L_on_W} is onto.
Take an arbitrary $g\in W_N^m\subseteq H^m$. Let $f:=Q(g)$ where $Q$ is defined by
formula \eqref{eq:Q}. It follows from Lemma \ref{lem:Q} $(iii)$ that $f=Q(g)\in H^{m+2}$ and
\begin{equation}\label{eq:f-relation}
(1-\partial_x^2)f=g\,.
\end{equation}
Lemma \ref{lem:Q} $(ii)$ implies that for any $0\le j\le m$,
\[
f^{(j)}=Q(g)^{(j)}=\frac{1}{2}\big(Q_-(g^{(j)})+Q_+(g^{(j)})\big),
\]
where the operators $Q_\pm$ are defined in \eqref{eq:Qpm}.
This together with  Lemma \ref{lem:Q} $(i)$ then gives that
\begin{equation}\label{eq:equation2}
f\in W_N^m\,.
\end{equation}
Further, in view of \eqref{eq:Q} and \eqref{eq:LQpm} we have
\begin{eqnarray}
f'&=&\frac{1}{2}\big(Q_+(g)'+Q_-(g)'\big)=\frac{1}{2}\big((g-Q_+(g))+(Q_-(g)-g)\big)\nonumber\\
&=&\frac{1}{2}\big(Q_-(g)-Q_+(g)\big).\label{eq:trick}
\end{eqnarray}
Hence,
\[
f^{(m+1)}=Q_-(g^{(m)})-Q_+(g^{(m)})\,.
\]
As $\x^{N+m} g^{(m)}\in L^2$ we obtain from  Lemma \ref{lem:Q} $(i)$ that
\begin{equation}\label{eq:equation3}
\x^{N+m} f^{(m+1)}\in L^2\,.
\end{equation}
It follows from \eqref{eq:f-relation} that
\[
f^{(m+2)}=(f'')^{(m)}=f^{(m)}-g^{(m)}\,.
\]
As $f,g\in W_N^m$ we get that
\begin{equation}\label{eq:equation4}
\x^{N+m} f^{(m+2)}\in L^2\,.
\end{equation}
Combining \eqref{eq:equation2},  \eqref{eq:equation3}, and  \eqref{eq:equation4} we conclude that
$f\in {\widetilde W}_N^{m+2}$, and therefore \eqref{eq:L_on_W} is onto.
Finally, the Lemma follows from the open mapping theorem.
\end{proof}

As a corollary of Lemma \ref{lem:L_on_W} we obtain
\begin{Prop} \label{prop:L_on_A}
For any $m\ge 1$, $N\ge 0$, and $n\ge 0$, the mapping,
\begin{equation}\label{eq:L_on_A}
\Lambda : \tilde\A_{n,N}^{m+2}\to \A_{n,N}^m,
\end{equation}
is a linear isomorphism.
\end{Prop}
\begin{proof}
For the sake of further reference we first record that for any  $k\ge 0$,
\begin{equation}\label{eq:'1}
\begin{array}{ccl}
\Big(\frac{1}{\x^k}\Big)'&=&-k\frac{x}{\x^{k+2}},\\
\Big(\frac{1}{\x^k}\Big)''&=&k(k+1)\frac{1}{\x^{k+2}}-k(k+2)\frac{1}{\x^{k+4}}
\end{array}
\end{equation}
and
\begin{equation}\label{eq:'2}
\begin{array}{ccl}
\Big(\frac{x}{\x^{k+1}}\Big)'&=&-k\frac{1}{\x^{k+1}}+(k+1)\frac{1}{\x^{k+3}},\\
\Big(\frac{x}{\x^{k+1}}\Big)''&=&k(k+1)\frac{x}{\x^{k+3}}-(k+1)(k+3)\frac{x}{\x^{k+5}}\,.
\end{array}
\end{equation}
As a consequence, for any $k\ge N+1$ and for any $r\ge 1$,
\begin{equation}\label{eq:1/x_in_W}
\frac{1}{\x^k},\;\frac{x}{\x^{k+1}}\in {\widetilde W}_N^{r+2}\subseteq W_N^r\,.
\end{equation}
This implies that for any $k\ge 0$ and for any $r\ge 1$,
\begin{equation}\label{eq:1/x_in_A}
\frac{1}{\x^k},\;\frac{x}{\x^{k+1}}\in\tilde\A_N^{r+2}\subseteq\A_N^{r+2}\,.
\end{equation}
Take $u\in\tilde\A_{n,N}^{m+2}$,
\[
u=\sum_{k=n}^N\Big(a_k\frac{1}{\x^k}+b_k\frac{x}{\x^{k+1}}\Big)+f,\;\;\;f\in{\widetilde W}_N^{m+2}\,.
\]
In view of the splitting \eqref{eq:splitting}, \eqref{eq:'1}, and \eqref{eq:'2}, we have,
\begin{equation}\label{eq:L}
\Lambda(u)=\Lambda_1(u)\oplus \Lambda_2(u)
\end{equation}
where
\begin{eqnarray}
\Lambda_1(u)&:=&\sum_{k=n}^{N-2}\Big(a_k \Lambda\Big(\frac{1}{\x^k}\Big)+
b_k \Lambda\Big(\frac{x}{\x^{k+1}}\Big)\Big)\nonumber\\
&-&\sum_{k=N-3}^{N-2}\Big(k(k+2)a_k\frac{1}{\x^{k+4}}+(k+1)(k+3)b_k\frac{x}{\x^{k+5}}\Big)\label{eq:L1}\\
&+&\sum_{k=N-1}^N\Big(a_k\frac{1}{\x^k}+b_k\frac{x}{\x^{k+1}}\Big)\nonumber
\end{eqnarray}
and
\begin{eqnarray}
\Lambda_2(u)&:=&-\sum_{k=N-1}^N\Big(a_k\frac{1}{\x^k}+b_k\frac{x}{\x^{k+1}}\Big)''\nonumber\\
&+&\sum_{k=N-3}^{N-2}\Big(k(k+2)a_k\frac{1}{\x^{k+4}}+(k+1)(k+3)b_k\frac{x}{\x^{k+5}}\Big)\label{eq:L2}\\
&+&\Lambda(f)\,.\nonumber
\end{eqnarray}
Here and below we use the convention that a sum vanishes if its upper limit is strictly less than its lower one.
It follows from \eqref{eq:L}-\eqref{eq:L2} and Lemma \ref{lem:L_on_W} that \eqref{eq:L_on_A}
is continuous. Moreover,  \eqref{eq:L_on_A} is injective by Lemma \ref{lem:Q} $(iv)$: 
by \eqref{eq:Q-bijective}, $\Lambda(u)=0$  implies $u=Q(\Lambda(u))=0$.

Next, we prove that \eqref{eq:L_on_A} is onto. Take an arbitrary $v\in\A_{n,N}^m$,
\begin{equation}\label{eq:v}
v=\sum_{k=n}^N\Big(a_k\frac{1}{\x^k}+b_k\frac{x}{\x^{k+1}}\Big)+g,\;\;\;g\in W_N^m\,.
\end{equation}
For any $k\ge 0$ and $l\ge 0$ we have
\begin{eqnarray}\label{eq:1/x}
\frac{1}{\x^k}&=&\Lambda\Big(\frac{1}{\x^k}\Big)+\Big(\frac{1}{\x^k}\Big)''
=\Lambda\Big(\frac{1}{\x^k}\Big)+\Big(\Lambda\Big(\frac{1}{\x^k}\Big)+\Big(\frac{1}{\x^k}\Big)''\Big)''\nonumber\\
&=&\Lambda\Big(\frac{1}{\x^k}+\Big(\frac{1}{\x^k}\Big)''\Big)+\Big(\frac{1}{\x^k}\Big)^{(4)}=...\nonumber\\
&=&\Lambda\Big(\sum_{j=0}^l\Big(\frac{1}{\x^k}\Big)^{(2j)}\Big)+\Big(\frac{1}{\x^k}\Big)^{(2l+2)}
\end{eqnarray}
and, similarly,
\begin{eqnarray}\label{eq:1/x'}
\frac{x}{\x^{k+1}}=\Lambda\Big(\sum_{j=0}^l\Big(\frac{x}{\x^{k+1}}\Big)^{(2j)}\Big)+
\Big(\frac{x}{\x^{k+1}}\Big)^{(2l+2)}\,.
\end{eqnarray}
Taking $l\ge 0$ such that $k+(2l+2)\ge N+1$ we get from \eqref{eq:'1}-\eqref{eq:1/x_in_W} that
\begin{equation}\label{eq:1/x''}
\Big(\frac{1}{\x^k}\Big)^{(2l+2)}, \Big(\frac{x}{\x^{k+1}}\Big)^{(2l+2)}\in W_N^m\,.
\end{equation}
Hence, in view \eqref{eq:1/x}, \eqref{eq:1/x'}, \eqref{eq:1/x''} and Lemma \ref{lem:L_on_W}, and \eqref{eq:1/x_in_A},
for any $k\ge 0$ there exists $u_k, w_k\in\tilde\A_{n,N}^{m+2}$ such that
\[
\frac{1}{\x^k}=\Lambda(u_k)\;\;\;\mbox{and\;\;\;}\frac{x}{\x^{k+1}}=\Lambda(w_k)\,.
\]
This together with Lemma \ref{lem:L_on_W} and \eqref{eq:v} implies that there exists $u\in\tilde\A_{n,N}^{m+2}$
such that
\[
v=\Lambda(u)\,.
\]
Finally, the Proposition follows from the open mapping theorem.
\end{proof}

Combining Proposition \ref{prop:L_on_A} with \eqref{eq:inclusion3} we get
\begin{Coro}\label{coro:L^{-1}}
For any $m\ge 3$, $N\ge 0$, and $n\ge 0$, the mapping,
\[
(1-\partial_x^2)^{-1} : \A_{n,N+2}^{m-2}\to\A_{n,N}^m\,,
\]
is well-defined and continuous.
\end{Coro}

\subsection{Smoothness of the conjugate map}\label{sec:smoothness_conjugate_map}
Assume that $m\ge 3$, $N\ge 0$, and $n\ge 0$. In this section we study the {\em conjugate map}
\begin{equation} \label{eq:alpha}
(\varphi,v)\stackrel{\sigma}{\mapsto}(\varphi,R_\varphi\circ(1-\partial_x^2)\circ R_{\varphi^{-1}}(v)),\;\;\;
\A\D_{n,N}^m\times\tilde\A_{n,N+2}^m\stackrel{\sigma}{\to}\A\D_{n,N}^m\times\A_{n,N+2}^{m-2}
\end{equation}
where $\varphi^{-1}$ denotes the inverse element of $\varphi$ in $\A\D_{n,N}^m$,
$R_\varphi(v):=v\circ\varphi$, and the symbol $\circ$ denotes the composition of mappings.
\begin{Lem}\label {lem:alpha}
For any $m\ge 3$, $N\ge 0$, and $n\ge 0$, 
the mapping \eqref{eq:alpha} is well-defined and $C^\infty$-smooth.
\end{Lem}
\begin{proof}
Let $(\varphi,v)\in\A\D_{n,N}^m\times\tilde\A_{n,N+2}^m$.
As $m\ge 3$, one gets from Lemma \ref{lem:decay} that the $C^1$-diffeomorphism $\varphi : \R\to\R$
satisfies $\varphi'(x)=1+o(1)$ as $|x|\to\infty$. This implies
\begin{equation}
0<\inf_{x\in\R}\varphi'(x)<\infty\,.
\end{equation}
As $(\varphi^{-1})'=1/\varphi'\circ\varphi^{-1}$ we obtain from Lemma \ref{lem:chain_rule} that
\[
R_\varphi\circ\partial_x\circ R_{\varphi^{-1}}(v)=
R_\varphi\Big(v'\circ\varphi^{-1}\cdot\frac{1}{\varphi'\circ\varphi^{-1}}\Big)
\]
and
\begin{eqnarray*}
R_\varphi\circ\partial_x^2\circ R_{\varphi^{-1}}(v)&=&
R_\varphi\left(v''\circ\varphi^{-1}\cdot\Big(\frac{1}{\varphi'\circ\varphi^{-1}}\Big)^2\right.\\
&-&\left.v'\circ\varphi^{-1}\cdot\varphi''\circ\varphi^{-1}\cdot\Big(\frac{1}{\varphi'\circ\varphi^{-1}}\Big)^3\right)\\
&=&v''\cdot\Big(\frac{1}{\varphi'}\Big)^2-v'\cdot\varphi''\cdot\Big(\frac{1}{\varphi'}\Big)^3.
\end{eqnarray*}
Hence,
\begin{equation}\label{eq:L-conjugate}
R_\varphi\circ(1-\partial_x^2)\circ R_{\varphi^{-1}}(v)=
v-v''\cdot\Big(\frac{1}{\varphi'}\Big)^2+v'\cdot\varphi''\cdot\Big(\frac{1}{\varphi'}\Big)^3\,.
\end{equation}
Now we claim that
\begin{equation}\label{eq:spread1}
(\varphi,v)\mapsto v''\cdot\Big(1/\varphi'\Big)^2,\;\;\A\D^m_{n,N}\times\tilde\A^m_{n,N+2}\to\A_{n,N+2}^{m-2},
\end{equation}
and
\begin{equation}\label{eq:spread2}
(\varphi,v)\mapsto v'\cdot\varphi''\cdot\Big(1/\varphi'\Big)^3,\;\;
\A\D^m_{n,N}\times\tilde\A^m_{n,N+2}\to\A_{n,N+2}^{m-2},
\end{equation}
are both $C^\infty$-smooth. In fact, let us first verify that
\begin{equation}\label{eq:phi-inverse}
\varphi\mapsto 1/\varphi',\;\;\A\D^m_{n,N}\to\A_{0,N+1}^{m-1},
\end{equation}
is $C^\infty$-smooth. This follows from Lemma \ref{lem:1/phi'} as $\varphi'=1+u'>0$,
$u\in\A^m_{n,N}$, and as the linear map $u\mapsto u'$, $\A^m_{n,N}\to\A^{m-1}_{n+1,N+1}$,
is continuous by Lemma \ref{lem:A-properties}. Note that any continuous linear (or multilinear) 
map is $C^\infty$-smooth. In particular, by the continuity of the pointwise product in
$\A_{0,N+1}^{m-1}$ we conclude from \eqref{eq:phi-inverse} that
\begin{equation}\label{eq:1/phi'^k}
\varphi\mapsto\Big(1/\varphi'\Big)^2,\;\;\A\D^m_{n,N}\to\A_{0,N+1}^{m-1},\,\,\,\mbox{and}\,\,\,
\varphi\mapsto\Big(1/\varphi'\Big)^3,\;\;\A\D^m_{n,N}\to\A_{0,N+1}^{m-1},
\end{equation}
are $C^\infty$-smooth.
Now to show that the map \eqref{eq:spread1} is $C^\infty$-smooth, we use that the inclusion
$\tilde\A_{n,N+2}^m\subseteq\A_{n,N}^m$ is continuous (see \eqref{eq:inclusion3}) to conclude 
that
\begin{equation}\label{eq:v->v''}
v\mapsto v'',\,\,\,\tilde\A_{n,N+2}^m\to\A_{n+2,N+2}^{m-2}\subseteq\A_{n+1,N+2}^{m-2},
\end{equation}
is smooth. Then the smoothness of \eqref{eq:spread1} follows from the smoothness of the maps
\eqref{eq:1/phi'^k}, \eqref{eq:v->v''}, and the continuity of the pointwise
product (Lemma \ref{lem:A-properties}),
$(f,g)\mapsto f\cdot g, \,\,\,\A_{n+1,N+2}^{m-2}\times\A_{0,N+1}^{m-1}\to\A^{m-2}_{n,N+2}$.

Similarly, we use the boundedness of the inclusion $\tilde\A^m_{n,N+2}\subseteq\A_{n,N+1}^{m-1}$, to obtain
the smoothness of  $v\mapsto v'$, $\tilde\A^m_{n,N+2}\to\A_{n+1,N+2}^{m-2}$, which together with the
smoothness of the maps
$\varphi\mapsto\varphi''$, $\A\D^m_{n,N}\to\A_{n+2,N+2}^{m-2}\subseteq\A_{0,N+2}^{m-2}$, and 
\eqref{eq:1/phi'^k}, implies that \eqref{eq:spread2} is $C^\infty$-smooth.

Finally, in view of the continuity of the inclusion $\tilde\A_{n,N+2}^m\subseteq\A_{n,N+2}^{m-2}$, one gets
from \eqref{eq:L-conjugate}, \eqref{eq:spread1}, and \eqref{eq:spread2} that
\eqref{eq:alpha} is $C^\infty$-smooth.
\end{proof}

The main result of this subsection is the following Proposition.
\begin{Prop}\label{prop:A}
For any $m\ge 3$, $N\ge 0$, and $n\ge 0$, there exists an open neighborhood $\U$ of the identity $\id$ in 
$\A\D_{n,N}^m$ such that the restriction of the map \eqref{eq:alpha} to  $\U\times\tilde\A_{n,N+2}^m$
is a $C^\infty$-diffeomorphism, i.e. if $\mathcal{C}:=\sigma|_{\U\times\tilde\A_{n,N+2}^m}$ then
\begin{equation}\label{eq:A}
\mathcal{C} : \U\times\tilde\A_{n,N+2}^m\to\U\times\A_{n,N+2}^{m-2},
\end{equation}
is a $C^\infty$-diffeomorphism.
\end{Prop}
\begin{proof}
By Lemma \ref{lem:alpha} the map \eqref{eq:alpha} is $C^\infty$-smooth.
The differential of $\sigma$ at the point $(\id,0)$, 
$d_{(\id,0)}\sigma : \A_{n,N}^m\times\tilde\A_{n,N+2}^m\to\A_{n,N}^m\times\A_{n,N+2}^{m-2}$,
is given by
\begin{equation}\label{eq:d_alpha}
(\delta\varphi,\delta v)\mapsto(\delta\varphi, (1-\partial_x^2)\delta v)\,.
\end{equation}
Then, it follows from Proposition \ref{prop:L_on_A} that \eqref{eq:d_alpha} is a linear isomorphism.
Hence, by the inverse function theorem in Banach spaces, there exist an open neighborhood $\U$ of the
identity $\id$ in $\A\D_{n,N}^m$ and an open neighborhood $\V$ of the zero in $\tilde\A_{n,N+2}^m$
such that
\begin{equation}\label{eq:A-premature}
\sigma_{|_{\U\times\V}} : \U\times\V\to\U\times\A_{n,N+2}^{m-2}
\end{equation}
is a $C^\infty$-diffeomorphism onto its image. Here we used that $\sigma(\varphi,0)=(\varphi,0)$
for any $\varphi\in\A\D_{n,N}^m$.
Note also that for any $(\varphi,v)\in \A\D_{n,N}^m\times\tilde\A_{n,N+2}^m$, 
\[
\pi_1\circ\sigma(\varphi,v)=\varphi\;\;\mbox{and}\;\;\pi_2\circ\sigma(\varphi,v)=
R_\varphi\circ(1-\partial_x^2)\circ R_{\varphi^{-1}}(v)\,,
\]
where $\pi_1 : \A\D_{n,N}^m\times\A_{n,N+2}^{m-2}\to\A\D_{n,N}^m$ and
$\pi_2 : \A\D_{n,N}^m\times\A_{n,N+2}^{m-2}\to\A_{n,N+2}^{m-2}$ denote the cartesian projections
onto the first and the second component. This and the fact that \eqref{eq:A-premature} is a 
$C^\infty$-diffeomorphism onto its image imply that for any $\varphi\in\U$ the linear mapping
$\pi_2\circ\sigma(\varphi,\cdot) : \V\to\A_{n,N+2}^{m-2}$ continuous, injective, and maps $\V$ onto
an open neighborhood of zero in $\A_{n,N+2}^{m-2}$.  This and the open mapping theorem then imply that
\begin{equation}\label{eq:conjugate_isomorphism}
\delta v\mapsto R_\varphi\circ(1-\partial_x^2)\circ R_{\varphi^{-1}}(\delta v),\;\;\;
\tilde\A_{n,N+2}^m\to\A_{n,N+2}^{m-2}\,,
\end{equation}
is a linear isomorphism of Banach spaces. In particular, the mapping
\begin{equation}\
\mathcal{C}=\sigma_{|_{\U\times\tilde\A_{n,N+2}^m}} : \U\times\tilde\A_{n,N+2}^m\to\U\times\A_{n,N+2}^{m-2}\,,
\end{equation}
is bijective.

Finally, by computing the differential of $\mathcal{C}$ at an arbitrary point  
$(\varphi,v)\in\V\times\tilde\A_{n,N+2}^m$ we obtain that
\begin{equation}\label{eq:dA}
d_{(\varphi,v)}\mathcal{C}(\delta\varphi,\delta v)=
\left[
\begin{array}{cc}
\id_{\A_{n,N}^m}&0\\
*&R_\varphi\circ(1-\partial_x^2)\circ R_{\varphi^{-1}}
\end{array}
\right]
\left[
\begin{array}{c}
\delta\varphi\\
\delta v
\end{array}
\right]
\end{equation}
where $\id_{\A_{n,N}^m} : \A_{n,N}^m\to\A_{n,N}^m$ is the identity in $\A_{n,N}^m$
and $*$ denotes a bounded linear map $\A_{n,N}^m\times\tilde\A^m_{n,N+2}\to\A_{n,N+2}^{m-2}$.
As \eqref{eq:conjugate_isomorphism} is a linear isomorphism we conclude from \eqref{eq:dA} that
\[
d_{(\varphi,v)}\mathcal{C} : \A_{n,N}^m\times\tilde\A_{n,N+2}^m\to \A_{n,N}^m\times\A_{n,N+2}^{m-2}
\]
is a linear isomorphism. Applying the inverse function theorem we get that $\mathcal{C}$ is a local
$C^\infty$-diffeomorphism. As $\mathcal{C}$ is bijective we conclude that it is a $C^\infty$-diffeomorphism.
\end{proof}

\subsection{Smoothness of the vector field}\label{sec:smoothness_vector_field}
Take $b\in\R$, $m\ge 3$, and $N\ge 0$, and $n\ge 0$. Here we consider the mapping,
\begin{equation}\label{eq:B}
(\varphi,v)\stackrel{{\mathcal B}_b}{\mapsto}(\varphi, R_\varphi\circ\beta_b\circ R_{\varphi^{-1}}(v)),\;\;
\A\D_{n,N}^m\times\A_{n,N}^m\stackrel{{\mathcal B}_b}{\to}\A\D_{n,N}^m\times\A_{n,N+2}^{m-2}\,,
\end{equation}
where 
\[
\beta_b(u):=-b\,u u_x-(3-b)\,u_x u_{xx}\,.
\]
First we prove the following Proposition.
\begin{Prop}\label{prop:B}
For any $n\ge 1$ the mapping \eqref{eq:B} is well-defined and $C^\infty$-smooth.
\end{Prop}
\begin{proof}
We follow the lines of the proof of Lemma \ref{lem:alpha}.
For any $(\varphi,v)\in\A\D_{n,N}^m\times\A_{n,N}^m$ one has in view of Lemma \ref{lem:chain_rule},
\begin{eqnarray}
{\mathcal B}_b(\varphi,v)=R_\varphi\circ\Big(-b v\circ\varphi^{-1}\cdot (v\circ\varphi^{-1})'-
(3-b)\,(v\circ\varphi^{-1})'\cdot (v\circ\varphi^{-1})''\Big)\nonumber\\
=-b\,\,v\cdot v'\cdot\frac{1}{\varphi'}+(3-b)\,(v')^2\cdot\varphi''\cdot\Big(\frac{1}{\varphi'}\Big)^4-
(3-b)\,v'\cdot v''\cdot\Big(\frac{1}{\varphi'}\Big)^3\label{eq:B-long}.
\end{eqnarray}
Using Lemma \ref{lem:A-properties} and Lemma \ref{lem:1/phi'} we get,
\[
v\in\A_{n,N}^m,\;\;\frac{1}{\varphi'}\in\A_{0,N+1}^{m-1},\;\;\mbox{and}\;\;v'\in\A_{n+1,N+1}^{m-1}\,.
\]
Hence, in view of Lemma \ref{lem:A-properties}, $v\cdot\frac{1}{\varphi'}\in\A_{n,N}^{m-1}$, and therefore,
\[
v'\cdot\Big(v\cdot\frac{1}{\varphi'}\Big)\in\A_{2n+1,N+n+1}^{m-1}\subseteq\A_{n,N+2}^{m-2}\,,
\] 
where we used that $n\ge 1$. This combined with the continuity of the pointwise product 
(Lemma \ref{lem:A-properties}) implies that
\[
(\varphi,v)\mapsto v'\cdot\Big(v\cdot\frac{1}{\varphi'}\Big),\;\;\;
\A\D_{n,N}^m\times\A_{n,N}^m\to\A_{n,N+2}^{m-2}
\]
is $C^\infty$-smooth. The other terms in \eqref{eq:B-long} are treated similarly.
\end{proof}

\begin{Rem}\label{rem:constant_coefficients}
The proof of Proposition \ref{prop:B} shows that 
\[
{\mathcal B}_b\big(\A\D_{n,N}^m\times\A_{n,N}^m\big)\subseteq\A\D_{n,N}^m\times\A^{m-2}_{2n+1,N+2}\,.
\]
\end{Rem}

Finally, combining Proposition \ref{prop:A}, Proposition \ref{prop:B}, and the fact that the inclusion,
\[
\tilde\A_{2n+1,N+2}^m\subseteq\A_{2n+1,N}^m\,,
\]
is continuous, we see that for $n\ge 1$ the mapping,
\begin{equation}\label{eq:F}
\mathcal{C}^{-1}\circ {\mathcal B}_b\big|_{\U\times\A_{n,N}^m} : \U\times\A_{n,N}^m\to\U\times\A_{2n+1,N}^m
\subseteq\U\times\A_{n,N}^m\,,
\end{equation}
where $\U$ is the open neighborhood given by Proposition \ref{prop:A} is $C^\infty$-smooth. 
In particular, we see that the mapping,
\begin{equation}\label{eq:F-long}
\begin{array}{c}
(\varphi,v)\stackrel{\F_b}{\mapsto}
\Big(v,R_\varphi\circ(1-\partial_x^2)^{-1}\circ\beta_b\circ R_{\varphi^{-1}}(v)\Big),\\
\U\times\A_{n,N}^m\stackrel{\F_b}{\to}\A_{n,N}^m\times\A_{n,N}^m,
\end{array}
\end{equation}
is $C^\infty$-smooth.
Hence, we proved the following Theorem.
\begin{Th}\label{th:F-smooth}
For any $m\ge 3$, $N\ge 0$, and $n\ge 1$, the mapping \eqref{eq:F-long} is $C^\infty$-smooth.
In addition, $\F_b\big(\U\times\A_{n,N}^m\big)\subseteq\A_{n,N}^m\times\A_{2n+1,N}^m$.
\end{Th}
\noindent
In view of Theorem \ref{th:F-smooth}, we see that $\F_b$ defines a $C^\infty$-smooth vector field on 
$\U\times\A_{n,N}^m\subseteq\A\D_{n,N}^m\times\A_{n,N}^m$.

\subsection{Proof of Theorem \ref{th:main1}}
In view of Theorem \ref{th:F-smooth} the right-hand side of the first equation in \eqref{eq:lagrangian_coordinates}
is a $C^\infty$-smooth vector field on $\U\times\A_{n,N}^m$. Hence, by the existence, uniqueness,
and continuous (or smooth) dependence on parameters theorem for ODE's in Banach spaces \cite{Lang},
there exists $R'>0$ and $T'>0$ so that for any $u_0$ in the ball $B_{\A^m_{n,N}}(R')$ of radius $R'$ in
$\A^m_{n,N}$ centered at the origin, there exists a unique solution $(\varphi,v)$ of 
\eqref{eq:lagrangian_coordinates} in $C^1([0,T'],\U\times\A^m_{n,N})$. 
This solution  depends continuously on the initial data $u_0\in B_{\A^m_{n,N}}(R')$, 
in the sense that the data-to-solution map,
\[
u_0\mapsto(\varphi,v),\,\,\,B_{\A^m_{n,N}}(R')\to C^1([0,T'],\U\times\A^m_{n,N}),
\]
is continuous. In view of Proposition \ref{prop:lagrangian_coordinates}, the curve $u:=v\circ\varphi^{-1}$
is the unique solution of \eqref{eq:b-family} in $C^0([0,T'],\A^m_{n,N})\cap C^1([0,T'],\A^{m-1}_{n,N})$ and,
by item $1)$ of Theorem \ref{th:asymptotic_group}, $u$ depends continuously on the initial data 
$u_0\in B_{\A^m_{n,N}}(R')$, in the sense that the data-to-solution map,
\[
u_0\mapsto u,\,\,\,B_{\A^m_{n,N}}(R')\to C^0([0,T'],\A^m_{n,N})\cap C^1([0,T'],\A^{m-1}_{n,N}),
\]
is continuous.
Next, recall that the solutions of \eqref{eq:b-family} possess the following
scaling invariance: {\em if $u(t)$, $t\in[0,\tau]$, $\tau>0$, is a solution of \eqref{eq:b-family} then
for any $\lambda>0$, $u_\lambda(t):=\lambda u(\lambda t)$, $t\in[0,\tau/\lambda]$, is a solution of 
the first equation in \eqref{eq:b-family} so that $u|_{t=0}=\lambda u_0$.}
Now, take an arbitrary $R>R'$ and denote $T:=\mu T'$ where $\mu:=R'/R$.
Let $w_0\in B_{\A^m_{n,N}}(R)$. Then, $u_0:=\mu w_0\in B_{\A^m_{n,N}}(R')$ and hence, there exists
a unique solution $u$ of equation \eqref{eq:b-family} in $C^0([0,T'],\A^m_{n,N})\cap C^1([0,T'],\A^{m-1}_{n,N})$.
By the scaling invariance, $w(t):=u(t/\mu)/\mu$ is a solution of the first equation in \eqref{eq:b-family} in 
$C^0([0,T],\A^m_{n,N})\cap C^1([0,T],\A^{m-1}_{n,N})$ so that $w|_{t=0}=w_0$. 
The solution $w$ is necessarily unique.
Otherwise, using scaling invariance, we will obtain that $u$ is not unique.
The same argument also shows that the data-to-solution map $w_0\mapsto w$,
$B_{\A^m_{n,N}}(R)\to C^0([0,T],\A^m_{n,N})\cap C^1([0,T],\A^{m-1}_{n,N})$, is continuous.
This completes the proof of the first statement of Theorem \ref{th:main1}.

The second statement can be proved as follows: If $n\ge N+1$ then by convention \eqref{eq:convention1}
there are no asymptotic terms and the statement trivially holds. Assume that $0\le n\le N$ and let
$(\varphi,v)\in C^1([0,T],\A\D^m_{n,N}\times\A^m_{n,N})$ be the solution of \eqref{eq:lagrangian_coordinates}. 
Then, by Proposition \ref{prop:lagrangian_coordinates},
\[
u:=v\circ\varphi^{-1}
\]
is the solution of \eqref{eq:b-family} in $C^0([0,T],\A^m_{n,N})\cap C^1([0,T],\A^{m-1}_{n,N})$.
In view of the last statement of Theorem \ref{th:F-smooth}, the coefficients,
\begin{equation}\label{eq:coefficients}
a_k, b_k,\,\,\,\,n\le k\le n_*:=\min\{2n,N\},
\end{equation}
in the asymptotic expansion of $v$ are independent of $t\in[0,T]$.
Then, by Lemma \ref{lem:decay}, for any $t\in[0,T]$,
\begin{equation}\label{eq:v-expansion}
v(t)=\sum_{k=n}^{n_*}\Big(a_k\frac{1}{\x^k}+b_k\frac{x}{\x^{k+1}}\Big)+o\Big(\frac{1}{\x^{n_*}}\Big).
\end{equation}
By Lemma \ref{lem:decay} and Lemma \ref{lem:1/x-composition}, for any $t\in[0,T]$ and for any $n\le k\le n_*$,
\begin{equation}\label{eq:expansion1}
\frac{1}{\langle\cdot\rangle}\circ\varphi=\frac{1}{\x^k}+O\Big(\frac{1}{\x^{n+1+k}}\Big)=
\frac{1}{\x^k}+O\Big(\frac{1}{\x^{2n+1}}\Big).
\end{equation}
This and Lemma \ref{lem:decay} also imply that
\begin{eqnarray}\label{eq:expansion2}
\frac{x}{\x^{k+1}}\circ\varphi&=\Big(x+O\Big(\frac{1}{\x^n}\Big)\Big)
\Big(\frac{1}{\x^{k+1}}+O\Big(\frac{1}{\x^{n+2+k}}\Big)\Big)\nonumber\\
&=\frac{x}{\x^{k+1}}+O\Big(\frac{1}{\x^{2n+1}}\Big)\label{eq:composition2}
\end{eqnarray}
Hence, in view of \eqref{eq:v-expansion}, \eqref{eq:expansion1}, and \eqref{eq:expansion2},
for any $t\in[0,T]$,
\[
u=v\circ\varphi^{-1}=
\sum_{k=n}^{n_*}\Big(a_k\frac{1}{\x^k}+b_k\frac{x}{\x^{k+1}}\Big)+o\Big(\frac{1}{\x^{n_*}}\Big).
\]
This shows that the first $2(n_*-n+1)$ coefficients in the asymptotic expansion of $u$ coincide with
\eqref{eq:coefficients} and therefore they are independent of $t\in[0,T]$.

\section{Analysis on the asymptotic space $\AH_{n,N}^m$}\label{sec:AH}
In this section we prove Theorem \ref{th:main2}.

\subsection{Mapping Properties of $(1-\partial_x^2)^{-1}$ on ${\AH}_{n,N}^m$}\label{sec:mapping_properties2}
Take $m\ge 1$, $N\ge 0$, and $n\ge 0$.
The analogs of Lemma \ref{lem:L_on_W} and Proposition \ref{prop:L_on_A} are easier to state and prove
since we do not have to modify the spaces $H_N^{m+2}$ and $\AH_N^{m+2}$ in order to describe
the mapping properties of $\Lambda=1-\partial_x^2$.
\begin{Lem}\label{lem:L_on_H} 
For any $m\ge 1$, $N\ge 0$, and $n\ge 0$, the mapping
\begin{equation}\label{eq:L_on_H}
\Lambda : H_N^{m+2}\to H_N^m
\end{equation}
is a linear isomorphism.
\end{Lem}
\begin{proof}
Clearly \eqref{eq:L_on_H} is continuous, so we only need to verify it is injective and surjective. 
Similar to the proof of Lemma \ref{lem:L_on_W}, injectivity follows from $H_N^{m+2}\subset H^{m+2}$ and
the isomorphism $\Lambda : H^{m+2}\to H^m$, while the surjectivity follows from Lemma  \ref{lem:Q} $(i)$.
\end{proof}
\begin{Prop} \label{prop:L_on_H}
For any $m\ge 1$, $N\ge 0$, and $n\ge 0$, the mapping
\begin{equation}\label{eq:L_on_AH}
\Lambda : \AH_{n,N}^{m+2}\to \AH_{n,N}^m
\end{equation}
is a linear isomorphism.
\end{Prop}
\begin{proof} The continuity of \eqref{eq:L_on_AH} is clear from the splitting \eqref{eq:L}-\eqref{eq:L2} and
Lemma \ref{lem:L_on_H} and the injectivity follows from Lemma \ref{lem:Q} $(iv)$. The surjectivity follows by 
replacing $W_N^m$ by $H_N^m$ in \eqref{eq:v} - \eqref{eq:1/x''} and using Lemma \ref{lem:L_on_H}.
\end{proof}
\begin{Rem}
By allowing $m=0$ in the definition of the asymptotic space $\A^m_{n,N}$ (resp. $\AH^m_{n,N}$) and in 
the corresponding remainder space $W^m_N$ (resp. $H^m_N$) one can easily verify that the mapping properties in
Lemmas \ref{lem:L_on_W} and \ref{lem:L_on_H} and in Propositions \ref{prop:L_on_A} and \ref{prop:L_on_H} are
true for $m=0$. 
(In fact, since $\widetilde W^2_N=H^2_N$,  both results \ref{prop:L_on_A} and \ref{prop:L_on_H} coincide.) 
However, the reason that we have specified $m\geq 1$ is that it suffices for our main purpose -- the proof of
Theorems \ref{th:main1} and \ref{th:main2}.
\end{Rem}

\subsection{Smoothness of the conjugate map}\label{sec:smoothness_conjugate_map2}
Assume that $m\ge 3$ and consider the conjugate map
\begin{equation} \label{eq:alphaH}
(\varphi,v)\stackrel{\sigma}{\mapsto}(\varphi,R_\varphi\circ(1-\partial_x^2)\circ R_{\varphi^{-1}}(v)),\;\;\;
\AH\D_{n,N}^m\times\AH_{n,N}^m\stackrel{\sigma}{\to}\AH\D_{n,N}^m\times\AH_{n,N}^{m-2}
\end{equation}
where $\varphi^{-1}$ denotes the inverse  of $\varphi$ in $\AH\D_{n,N}^m$.
Arguing as in Section \ref{sec:smoothness_conjugate_map} one gets from Lemma \ref{lem:H-properties},
Lemma \ref{lem:AH-properties}, and Lemma \ref{lem:1/phi'}, the following Lemma.
\begin{Lem}\label {lem:alphaH}
For any $m\ge 3$, $N\ge 0$, and $n\ge 0$, the mapping \eqref{eq:alphaH} is well-defined and $C^\infty$-smooth.
\end{Lem}
As a consequence, we prove as in Section \ref{sec:smoothness_conjugate_map},
\begin{Prop}\label{prop:AH}
For any $m\ge 3$, $N\ge 0$, and $n\ge 0$, there exists an open neighborhood $\U$ of the identity $\id$ in 
$\AH\D_{n,N}^m$ such that the restriction of the map \eqref{eq:alphaH} to $\U\times\AH_{n,N}^m$ is $C^\infty$-smooth, i.e. 
if $C:=\sigma|_{\U\times\AH_{n,N}^m}$ then 
\begin{equation}\label{eq:AH}
C : \U\times\AH_{n,N}^m\to\U\times\AH_{n,N}^{m-2}
\end{equation}
is a $C^\infty$-diffeomorphism.
\end{Prop}

\subsection{Smoothness of the vector field}\label{sec:smoothness_vector_field2}
Take $b\in\R$ and consider the mapping,
\begin{equation}\label{eq:BH}
(\varphi,v)\stackrel{B_b}{\mapsto}(\varphi, R_\varphi\circ\beta_b\circ R_{\varphi^{-1}}),\;\;
\AH\D_{n,N}^m\times\AH_{n,N}^m\stackrel{B_b}{\to}\AH\D_{n,N}^m\times\AH_{n,N}^{m-2}\,,
\end{equation}
where $\beta_b(u)=-b\,u u_x-(3-b)\,u_x u_{xx}$.
As is Section \ref{sec:smoothness_vector_field} one has
\begin{Prop}\label{prop:BH}
For any $m\ge 3$, $N\ge 0$, and $n\ge 0$, the mapping \eqref{eq:BH} is well-defined and $C^\infty$-smooth.
\end{Prop}
\begin{Rem}
Note that, unlike in Proposition \ref{prop:B}, the case $n=0$ is {\em not} excluded in Proposition \ref{prop:BH}.
\end{Rem}
\begin{proof}
As in the proof of  Proposition \ref{prop:B} we see that for any $(\varphi,v)\in\AH\D_{n,N}^m\times\AH_{n,N}^m$,
\begin{eqnarray}
B_b(\varphi,v)=R_\varphi\circ\Big(-b\,v\circ\varphi^{-1}\cdot (v\circ\varphi^{-1})'-
(3-b)\,(v\circ\varphi^{-1})'\cdot (v\circ\varphi^{-1})''\Big)\nonumber\\
=-b\,v\cdot v'\cdot\frac{1}{\varphi'}+(3-b)\,(v')^2\cdot\varphi''\cdot\Big(\frac{1}{\varphi'}\Big)^4-
(3-b)\,v'\cdot v''\cdot\Big(\frac{1}{\varphi'}\Big)^3\label{eq:BH-long}.
\end{eqnarray}
By Lemma \ref{lem:AH-properties} and Lemma \ref{lem:1/phi'},
\[
v\in\AH_{n,N}^m,\;\;v'\in\AH_{n+1,N}^{m-1},\;\;\mbox{and}\;\;\frac{1}{\varphi'}\in\AH_{0,N}^{m-1}\,.
\]
In view of Lemma \ref{lem:AH-properties}, $v\cdot\frac{1}{\varphi'}\in\AH_{n,N}^{m-1}$, and hence,
\[
v'\cdot\Big(v\cdot\frac{1}{\varphi'}\Big)\in\AH_{2n+1,N+d}^{m-1}\subseteq\AH_{n,N}^{m-2}\,,
\] 
This combined with the continuity of the pointwise product (Lemma \ref{lem:AH-properties}) implies that
\begin{equation}\label{eq:BH1}
(\varphi,v)\mapsto v'\cdot\Big(v\cdot\frac{1}{\varphi'}\Big),\;\;\;
\AH\D_{n,N}^m\times\AH_{n,N}^m\to\AH_{2n+1,N}^{m-2}\subseteq\AH_{2n+1,N}^{m-2}
\subseteq\AH_{n,N}^{m-2}
\end{equation}
is $C^\infty$-smooth. The other terms in \eqref{eq:BH-long} are treated similarly.
In fact, as 
\[
v'\in\AH_{n+1,N}^{m-1},\;\;v''\in\AH_{n+2,N}^{m-2},\;\;\varphi''\in\AH_{n+2,N}^{m-2},\;\;
\mbox{and}\;\;\frac{1}{\varphi'}\in\AH_{0,N}^{m-1}\,,
\]
we see from the continuity of the pointwise product (Lemma \ref{lem:AH-properties}) that
\begin{equation}\label{eq:BH2}
(\varphi,v)\mapsto (v')^2\cdot\varphi''\cdot\Big(\frac{1}{\varphi'}\Big)^4,\;\;\;
\AH\D_{n,N}^m\times\AH_{n,N}^m\to\AH_{2n+3,N}^{m-2}\subseteq\AH_{n,N}^{m-2}
\end{equation}
and 
\begin{equation}\label{eq:BH3}
(\varphi,v)\mapsto v'\cdot v''\cdot\Big(\frac{1}{\varphi'}\Big)^3,\;\;\;
\AH\D_{n,N}^m\times\AH_{n,N}^m\to\AH_{2n+3,N}^{m-2}\subseteq\AH_{n,N}^{m-2}
\end{equation}
are $C^\infty$-smooth.
In conclusion, we obtain from \eqref{eq:BH-long}-\eqref{eq:BH3} that
\begin{equation}\label{eq:constant_leading_terms}
(\varphi,v)\mapsto B_b(\varphi,v),\;\;\;\AH\D_{n,N}^m\times\AH_{n,N}^m\to
\AH_{2n+1,N}^{m-2}\subseteq\AH_{n,N}^{m-2}
\end{equation}
is $C^\infty$-smooth.
\end{proof}
Now, arguing as in Section \ref{sec:smoothness_conjugate_map} we see that the mapping,
\begin{equation}\label{eq:FH-long}
\begin{array}{c}
(\varphi,v)\stackrel{F_b}{\mapsto}
\Big(v,R_\varphi\circ(1-\partial_x^2)^{-1}\circ\beta_b\circ R_{\varphi^{-1}}(v)\Big),\\
\U\times\AH_{n,N}^m\stackrel{F_b}{\to}\AH_{n,N}^m\times\AH_{2n+1,N}^m\subseteq\AH_{n,N}^m\times\AH_{n,N}^m,
\end{array}
\end{equation}
where $\U$ is the open neighborhood given by Proposition \ref{prop:AH} is $C^\infty$-smooth.
Hence, we proved the following Theorem.
\begin{Th}\label{th:FH-smooth}
For any $m\ge 3$, $N\ge 0$, and $n\ge 0$, the mapping \eqref{eq:FH-long} is $C^\infty$-smooth.
In addition, $F_b\big(\U\times\AH_{n,N}^m\big)\subseteq\AH_{n,N}^m\times\AH_{2n+1,N}^m$.
\end{Th}
\begin{Rem}
Note that, unlike in Theorem \ref{th:F-smooth}, the case $n=0$ is {\em not} excluded in
Theorem \ref{th:FH-smooth}.
\end{Rem}

\noindent
In view of Theorem \ref{th:FH-smooth}, we see that $F_b$ defines a $C^\infty$-smooth vector field on 
$\U\times\AH_{n,N}^m\subseteq\AH\D_{n,N}^m\times\AH_{n,N}^m$.

\subsection{Proof of Theorem \ref{th:main2}}
The proof of Theorem \ref{th:main2} follows by the same arguments as in the proof of
Theorem \ref{th:main1}.

\section{Appendix A: Lagrangian description}
In this Appendix we give the proofs of Lemma \ref{lem:ode} and
Proposition \ref{prop:lagrangian_coordinates}.

\medskip
\noindent{\bf Proof of Lemma \ref{lem:ode}.}
Let $u\in C^0([0,T],\A^m_{n,N})$. Denote $F(t,\varphi):=u(t)\circ\varphi$. Then by Theorem \ref{th:asymptotic_group},
\[
F : [0,T]\times\A D^{m-1}_{n,N}\to\A D^{m-1}_{n,N}
\]
and its partial derivative with respect to the second variable\footnote{Here 
${\mathcal L}(V,W)$ denotes the Banach space of bounded linear maps between two
Banach spaces $V$ and $W$ supplied with the uniform operator norm.}
\[
D_2F : [0,T]\times\A D^{m-1}_{n,N}\to {\mathcal L}\big(\A^{m-1}_{n,N},\A^{m-1}_{n,N}\big)
\]
are continuous. This implies that $F$ is {\em locally Lipschitz continuous} i.e. for any 
$(t_0,\varphi_0)\in[0,T]\times\A D^{m-1}_{n,N}$ there exists an open neighborhood $V$ of
$(t_0,\varphi_0)$ in $[0,T]\times\A D^{m-1}_{n,N}$ and $K>0$ such that for any $(t,\varphi_1)$ and
$(t,\varphi_2)$ in $V$,
\[
\|F(t,\varphi_2)-F(t,\varphi_1)\|_{\A^{m-1}_{n,N}}\le K\|\varphi_2-\varphi_1\|_{\A^{m-1}_{n,N}}\,.
\]
This together with the existence theorem for ODE's in Banach spaces \cite{Lang} implies that for
any $t_0\in[0,T]$ there exists an open neighborhood $U_{t_0}$ of $(t_0,\id)$ in $[0,T]\times\A D^{m-1}_{n,N}$
and $\varepsilon_{t_0}>0$ such that for any $(\tau,\psi)\in U_{t_0}$ there exists a unique solution 
$\varphi\in C^1\big([0,T]\cap(\tau-\varepsilon_{t_0},\tau+\varepsilon_{t_0}),\A D^{m-1}_{n,N}\big)$ of
$\dot\varphi=u\circ\varphi$, $\varphi|_{t=\tau}=\psi$. In view of the compactness of $[0,T]\times\id$ in 
$[0,T]\times\A D^{m-1}_{n,N}$ we see that there exists $\varepsilon>0$ such that for any $t_0\in[0,T]$ 
there exists a unique solution $\varphi\in C^1\big([0,T]\cap(t_0-\varepsilon,t_0+\varepsilon),\A D^{m-1}_{n,N}\big)$ of 
\begin{equation}\label{eq:ode_at_id}
\dot\varphi=u\circ\varphi,
\,\,\,\varphi|_{t=t_0}=\id\,.
\end{equation}
Note that if $\varphi\in C^1\big([0,T]\cap(t_0-\varepsilon,t_0+\varepsilon),\A D^{m-1}_{n,N}\big)$ is a solution
of \eqref{eq:ode_at_id} then for any $\psi\in\A D^{m-1}_{n,N}$ the curve $t\mapsto\varphi(t)\circ\psi$,
$\varphi\circ\psi\in C^1\big([0,T]\cap(t_0-\varepsilon,t_0+\varepsilon),\A D^{m-1}_{n,N}\big)$ is a solution of
\begin{equation}\label{eq:ode_general}
\dot\varphi=u\circ\varphi,
\,\,\,\varphi|_{t=t_0}=\psi\,.
\end{equation}
This solution is unique as $F(t,\varphi)=u(t)\circ\varphi$ is locally Lipschitz continuous on 
$[0,T]\times\A D^{m-1}_{n,N}$.
As $\varepsilon>0$ is independent of the choice of $t_0\in[0,T]$ and $\psi\in\A D^{m-1}_{n,N}$ we conclude that 
\eqref{eq:ode} has a unique solution 
\[
\varphi\in C^1\big([0,T],\A D^{m-1}_{n,N}\big)\,.
\]
As $\partial_x : \A^{m-1}_{n,N}\to\A^{m-2}_{n+1,N+1}$ is a bounded linear map we see from 
Lemma \ref{lem:chain_rule} $(i)$ that
\begin{equation}
(\varphi_x)^{\cdot}=u_x\circ\varphi\cdot\varphi_x
\end{equation}
where the $t$-derivative is understood in $\A^{m-2}_{n+1,N+1}$.
As $\varphi\in C^1\big([0,T],\A D^{m-1}_{n,N}\big)$ we see that 
$\varphi_x\in C^1\big([0,T],\A D^{m-2}_{n+1,N+1}\big)$. By the inclusion $\A^{m-2}_{n+1,N+1}\subseteq L^\infty$
one concludes that for any given $x\in\R$, $\varphi_x(\cdot,x)\in C^1([0,T],\R)$. This together with
$\varphi_x(t,x)>0$ implies that $(\log\varphi_x(t,x))^\cdot=u_x(t,\varphi(t,x))$ for any given $x\in\R$ and $t\in[0,T]$.
Hence, for any $x\in\R$ and for any $t\in[0,T]$,
\begin{eqnarray}
\varphi_x(t,x)&=&e^{\int_0^t\big(u_x(s)\circ\varphi(s)\big)(x)\,ds}\nonumber\\
&=&1+\sum_{k\ge 1}\Big(\int_0^t\big(u_x(s)\circ\varphi(s)\big)(x)\Big)/k!\,.\label{eq:exponent}
\end{eqnarray}
Note that by assumption $u_x\in C^0\big([0,T],\A^{m-1}_{n+1,N+1}\big)$. 
As $\A^{m-1}_{n+1,N+1}\subseteq \frac{1}{\x}\A^{m-1}_{n,N}$ and as $\varphi\in C^1\big([0,T],\A D^{m-1}_{n,N}\big)$ we
conclude from Theorem \ref{th:asymptotic_group}, Lemma \ref{lem:1/x-composition}, and Lemma \ref{lem:A-properties} that 
$u_x\circ\varphi\in C^0\big([0,T],\A^{m-1}_{n+1,N+1}\big)$.
This implies that
\[
\int_0^t u_x(s)\circ\varphi(s)\,ds\in\A^{m-1}_{n+1,N+1}
\]
as the integrand is a continuous curve in $\A^{m-1}_{n+1,N+1}$.
As $\A^{m-1}_{n+1,N+1}$ is a Banach algebra we conclude from \eqref{eq:exponent} that 
\[
\varphi_x(t)-1\in\A^{m-1}_{n+1,N+1}
\]
and
\[
\varphi_x-1\in C^1\big([0,T],\A^{m-1}_{n+1,N+1}\big).
\]
As in addition $\varphi\in C^1([0,T],\A D^{m-1}_{n,N})$
we see that 
\[
\varphi\in C^1\big([0,T],\A D^{m}_{n,N}\big)\,.
\] 
This solution is unique in $\A D^{m}_{n,N}$ as it is unique as a curve in $\A D^{m-1}_{n,N}$.
\hfill $\Box$

\medskip
\noindent{\bf Proof of Proposition \ref{prop:lagrangian_coordinates}.}
Let $u\in C^0([0,T],\A^m_{n,N})\cap C^1([0,T],\A^{m-1}_{n,N})$ be a solution of \eqref{eq:b-family}.
Then we have
\[
(1-\partial_x^2) (u_t+u u_x)=-b u u_x-(3-b) u_x u_{xx}\,.
\]
It follows from Lemma \ref{lem:A-properties} that $u_t+u u_x\in\A^{m-1}_{n,N}$ and
$\beta_b(u)\in\A^{m-2}_{n,N}$.
In view of Lemma \ref{lem:Q} $(iv)$, we get
\begin{equation}\label{eq:ch'}
u_t+u u_x=(1-\partial_x^2)^{-1}(\beta_b(u))
\end{equation}
where $(1-\partial_x^2)^{-1}f:=Q(f)$ and $Q(f)$ is defined by \eqref{eq:Q}.

Next, consider the differential equation
\begin{equation}\label{eq:ode_inside}
\dot\varphi=u\circ\varphi,\,\,\varphi|_{t=0}=\id\,.
\end{equation}
In view of Lemma \ref{lem:ode} there exists a unique solution 
\[
\varphi\in C^1([0,T],\A D^m_{n,N}).
\]
Denote
\[
v:=u\circ\varphi\,.
\]
It follows from Sobolev embedding theorem that $u(t,x)$ and $\varphi(t,x)$ are functions in
$C^1([0,T]\times\R,\R)$.
In particular, by \eqref{eq:ch'},
\begin{eqnarray*}
v_t&=&u_t\circ\varphi+u_x\circ\varphi\cdot\varphi_t\\
&=&(u_t+u u_x)\circ\varphi\\
&=&[(1-\partial_x^2)^{-1}(\beta_b(u))]\circ\varphi\,.
\end{eqnarray*}
Hence, for any given $x\in\R$ and $t\in[0,T]$,
\begin{equation}\label{eq:v_t}
v(t,x)=u_0(x)+\int_0^t\big([(1-\partial_x^2)^{-1}(\beta_b(u(s)))]\circ\varphi(s)\big)(x)\,dt\,.
\end{equation}
As $u\in C^0([0,T],\A^m_{n,N})$ we obtain from Lemma \ref{lem:A-properties} and $n\ge 1$ that 
\[
\beta_b(u)\in C^0([0,T],\A^{m-2}_{n,N+2})\,.
\]
In view of Corollary \ref{coro:L^{-1}} we get
\[
(1-\partial_x^2)^{-1}(\beta_b(u))\in C^0([0,T],\A^m_{n,N})\,.
\]
As $\varphi\in C^1([0,T],\A D^m_{n,N})$ we obtain from Theorem \ref{th:asymptotic_group} that
\[
[(1-\partial_x^2)^{-1}(\beta_b(u))]\circ\varphi\in C^0([0,T],\A^m_{n,N})\,.
\]
This implies that the integrand in \eqref{eq:v_t} converges in $\A^m_{n,N}$. Hence, 
\[
v\in C^1([0,T],\A^m_{n,N})
\]
and 
\[
\dot v=R_\varphi\circ(1-\partial_x^2)^{-1}\circ\beta_b\circ R_{\varphi^{-1}}(v)\,.
\]

\noindent Conversely, assume that 
\[
(\varphi,v)\in C^1\big([0,T],\A D^m_{n,N}\times\A^m_{n,N}\big)
\]
is a solution of \eqref{eq:lagrangian_coordinates}.
Denote 
\[
u:=v\circ\varphi^{-1}\,.
\]
In view of Theorem \ref{th:asymptotic_group},
\[
u\in C^0([0,T],\A^m_{n,N})\cap C^1([0,T],\A^{m-1}_{n,N})\,.
\]
By inspection one checks that $u$ is a solution of \eqref{eq:b-family}.

The fact that the described above correspondence between solutions of \eqref{eq:b-family} and
\eqref{eq:lagrangian_coordinates} is bijective follows from the fact that if
$u_1,u_2\in C^0([0,T],\A^{m}_{n,N})$, $u_1\ne u_2$ then the corresponding solutions
$\varphi_1,\varphi_2\in C^1([0,T],\A D^{m}_{n,N})$ of \eqref{eq:ode_inside} do not coincide.
\hfill $\Box$

\section{Appendix B: Auxiliary Results}\label{sec:auxiliary}
In this Appendix we collect for the convenience of the reader some auxiliary results needed in the main
body of the paper. For any $\gamma\in\R$ denote
\[
L^2_\gamma:=\{f\in L^2_{loc}\,|\, \x^\gamma f\in L^2 \}\,.
\]
Let
\begin{equation}\label{eq:L^2_*}
L^2_*:=\bigcup\limits_{\gamma\in\R} L^2_\gamma\,.
\end{equation}
For any $g\in L^2_*$ define the integral transforms
\begin{equation}\label{eq:Qpm}
(Q_\pm(g))(x):=\int_0^\infty g(x\mp z)\,e^{-z}\,dz
\end{equation}
and 
\begin{equation}\label{eq:Q}
Q(g):=\frac{1}{2}(Q_+(g)+Q_-(g))\,.
\end{equation}

\begin{Lem}\label{lem:Q}
$(i)$ For any $g\in L^2_*$, $Q_\pm(g)\in H^1_{loc}$ and
\begin{equation}\label{eq:LQpm}
(1\pm\partial_x)Q_\pm(g)=g\,.
\end{equation}
Moreover, for any $\gamma\in\R$,
\[
\x^\gamma f\in L^2\implies \x^\gamma Q_\pm(f)\in L^2\,.
\]
$(ii)$  For any $g\in H^1_{loc}\cap L^2_*$, $Q_\pm(g)\in H^2_{loc}$. If in addition $g'\in L^2_*$ then
\begin{equation}
Q_\pm(g)'=Q_\pm(g')\;\;\;\mbox{and}\;\;\;(1\pm\partial_x)Q_\pm(g)=Q_{\pm}\big((1\pm\partial_x) g\big)=g\,.
\end{equation}
The restriction of $Q_\pm$ to $H^1$ is a linear isomorphism $Q_\pm|_{H^1} : H^1\to H^2$.

\noindent $(iii)$ For any $g\in H^1_{loc}\cap L^2_*$, $Q(g)\in H^3_{loc}$, and
\begin{equation}
(1-\partial_x^2) Q(g)=g\,.
\end{equation}
If in addition $g'\in L^2_*$ then
\begin{equation}\label{eq:Q'}
Q(g)'=Q(g')\,.
\end{equation}
The restriction of $Q$ to $H^1$ is a linear isomorphism $Q|_{H^1} : H^1\to H^3$.

\noindent $(iv)$ Assume that $f\in H^2_{loc}$ and $f,f',f''\in L^2_*$. Then
\begin{equation}\label{eq:Q-bijective}
Q\big((1-\partial_x^2)f\big)=(1-\partial_x^2)Q(f)=f\,.
\end{equation}
\end{Lem}
\begin{proof}
$(i)$  Take $g\in L^2_*$. As $g\in L^2_\gamma$ for some $\gamma\in\R$ one sees that for any $a\in\R$ the
function 
$y\mapsto g(\mp y)\,e^{-y}=
\big(g(\mp y)\langle y\rangle^\gamma\big)\cdot\big(\,e^{-y}\langle y\rangle^{-\gamma}\big)$ is summable on
$[a,\infty)$  and
\begin{equation}\label{eq:Q-long}
(Q_\pm(g))(x)=\int_0^\infty g(x\mp z)\,e^{-z}\,dz=e^{\mp x}\int_{\mp x}^\infty g(\mp y)\,e^{-y}\,dy\,.
\end{equation}
This formula implies that $Q_\pm(g)$ is an {\em absolutely continuous} function on any finite interval
of $\R$ and its derivative\footnote{Recall that the (pointwise) derivative of an absolutely continuous
function exists a.e. and coincides with its weak derivative.} belongs to $L^2_{loc}$.
Differentiating \eqref{eq:Q-long} we see that,
\begin{equation}\label{eq:LQ2}
Q_\pm(g)'=\pm g \mp Q_{\pm}(g)\,,
\end{equation}
which implies \eqref{eq:LQpm}.
As $Q_\pm(g)\in L^2_{loc}$ we get from \eqref{eq:LQ2} that $Q_\pm(g)\in H^1_{loc}$.

Now, consider the second statement in $(i)$.
Take $\gamma\ge 0$ and note that for any $x,z\in\R$,
\begin{equation}\label{eq:inequality_weights}
\x^2\le 2\langle x\pm z\rangle^2\langle z\rangle^2\,.
\end{equation}
It follows from \eqref{eq:inequality_weights} and Young's inequality that for any 
measurable function $f$ such that $\x^\gamma f\in L^2$, $\x^\gamma Q_+(f)\in L^2$ and
\begin{eqnarray*}
\|\x^\gamma Q_+(f)\|_{L^2}&=&\Big\|\int_0^\infty  \x^\gamma f(x - z)\,e^{-z}\,dz\Big\|_{L^2}\\
&\le& 2^{\gamma/2}\Big\|\int_0^\infty \langle x - z\rangle^\gamma |f(x - z)|\,
\langle z\rangle^\gamma\,e^{-z}\,dz\Big\|_{L^2}\\
&=&2^{\gamma/2}\big\|\big(\langle\cdot\rangle^\gamma |f(\cdot)|\big)*
\big(\langle\cdot\rangle^\gamma\chi_{[0,\infty)}(\cdot)\,e^{-(\cdot)}\big)\big\|_{L^2}\le
C \|\x^\gamma f\|_{L^2},
\end{eqnarray*}
where $\chi_{[0,\infty)}$ is the characteristic function of $[0,\infty)$ and 
$C>0$ is independent of the choice of $f$. The case of the transformation $Q_-$
is considered in the same way. If $\gamma<0$ one uses the inequality 
$\langle x\pm z\rangle^2\le 2\x^2\langle z\rangle^2$ and argues similarly.

$(ii)$ As $g\in H^1_{loc}\cap L^2_*$ we see from $(i)$ that $Q_\pm(g)\in H^1_{loc}$ and
$Q_\pm(g)'=g-Q_{\pm}(g)$. This implies that $Q_\pm(g)\in H^2_{loc}$.

Assume in addition that $g'\in L^2_*$. Then, for any test function $\psi\in C^\infty_0$ such that
$\supp\psi\subseteq (-R,R)$, $0<R<\infty$,  we get from Fubini's theorem and an integration by parts that,
\begin{eqnarray}
\langle Q_\pm(g)',\psi\rangle&:=&-\int_{-\infty}^\infty\psi'(x)
\int_0^\infty g(x\mp z)\,e^{-z}\,dz\,dx\nonumber\\
&=&-\int_0^\infty e^{-z}\int_{-R}^R g(x\mp z)\psi'(x)\,dx\,dz\nonumber\\
&=&\int_0^\infty e^{-z}\int_{-R}^R g'(x\mp z)\psi(x)\,dx\,dz\nonumber\\
&=&\int_{-\infty}^\infty\Big(\int_0^\infty g'(x\mp z)\,e^{-z}\,dz\Big)\psi(x)\,dx\nonumber\\
&=&\langle Q_\pm(g'),\psi\rangle\label{eq:weak_nerivative}\,.
\end{eqnarray}
We can apply Fubini's theorem as by the first statement of $(i)$, 
$\int_0^\infty |g(x\mp z)|\,e^{-z}\,dz$ and $\int_0^\infty |g'(x\mp z)|\,e^{-z}\,dz$
are locally summable functions in the variable $x$.
Formula \eqref{eq:weak_nerivative} then implies that, 
\begin{equation}\label{eq:Qpm'2}
Q_\pm(g)'=Q_\pm(g')\,.
\end{equation}
This together with \eqref{eq:LQ2} implies that
\begin{equation}
Q_{\pm}\big((1\pm\partial_x) g\big)=(1\pm\partial_x)Q_\pm(g)=g\,.\label{eq:LQ3}
\end{equation}
Now, we prove the last statement in $(ii)$.
Assume that $g\in H^1$. Then $g, g'\in L^2\subseteq L^2_*$, and in
view of \eqref{eq:Qpm'2} and the second statement in $(i)$ 
(with $\alpha=0$) we see that $Q_\pm(g)'=Q_\pm(g')\in H^1$. This shows that $Q_\pm(g)\in H^2$.
Hence, the linear map $Q_\pm|_{H^1} : H^1\to H^2$ is well defined.
It follows from \eqref{eq:LQ3} that $Q_\pm|_{H^1} : H^1\to H^2$ is a bijective map with a continuous
inverse $(1\pm\partial_x) : H^2\to H^1$. 
This together with the open mapping theorem imply that $Q_\pm\big|_{H^1} : H^1\to H^2$ is
a linear isomorphism.

$(iii)$ This item follows immediately from $(i)$ and $(ii)$. In fact, for any $g\in H^1_{loc}\cap L^2_*$
one gets from \eqref{eq:Q} and $(ii)$ that $Q(g)\in H^2_{loc}$. Moreover,
\begin{eqnarray}
(1-\partial_x^2) Q(g)&=&\frac{1}{2}\Big((1-\partial_x)(1+\partial_x) Q_+(g)+
(1+\partial_x)(1-\partial_x) Q_-(g)\Big)\nonumber\\
&=&\frac{1}{2}\Big((1-\partial_x)g+(1+\partial_x)g\Big)=g\,.\label{eq:LQ_proof}
\end{eqnarray}
This implies that $Q(g)''=g-Q(g)\in H^1_{loc}$, and therefore $Q(g)\in H^3_{loc}$.

\noindent Assume that $g\in H^1$. Then it follows from \eqref{eq:Q} and the last statement in $(ii)$ that
$Q(g)\in H^2$. In view of \eqref{eq:LQ_proof}, $Q(g)''=g-Q(g)\in H^1$, and hence, $Q(g)\in H^3$. Hence, the
linear map $Q|_{H^1} : H^1\to H^3$ is well defined. Now, take $f\in H^3$. Then, by applying 
\eqref{eq:Qpm'2} twice we get
\[
Q\big((1-\partial_x^2)f\big)=(1-\partial_x^2)Q(f)=f\,.
\]
This implies that $Q|_{H^1} : H^1\to H^3$ is a bijective map with a continuous inverse
$(1-\partial_x^2) : H^3\to H^1$. Hence, by the open mapping theorem $Q|_{H^1} : H^1\to H^3$ 
a linear isomorphism.

\noindent The proof of $(iv)$ follows from $(iii)$.
\end{proof}

\begin{Lem}\label{lem:chain_rule}
$(i)$ Assume that $f\in H^1_{loc}$ and $\varphi : \R\to\R$ be an orientation preserving
$C^1$-diffeomorphism of $\R$. 
Then, $R_\varphi(f)=f\circ\varphi\in H^1_{loc}$ and
\[
(f\circ\varphi)'=f'\circ\varphi\cdot\varphi'\,.
\]
If $f\in H^1$ and $0<\inf\limits_{x\in\R}\varphi'(x)<\infty$ then $f\circ\varphi\in H^1$.

\noindent $(ii)$ Assume that $\psi:=1+f>0$
where $f\in H^1_{loc}$.\footnote{$\psi>0$ means that for any $x\in\R$, $\psi(x)>0$.} 
Then $1/\psi\in H^1_{loc}$ and
\[
\Big(\frac{1}{\psi}\Big)'=-\frac{1}{\psi^2}\cdot\psi'\,.
\]
If $f\in H^1$ and $0<\inf\limits_{x\in\R}\psi(x)$ then $(1/\psi)'\in L^2$.
\end{Lem}
\begin{proof}
As the both items of the Lemma follow in a similar way from a simple approximation argument, we prove only $(i)$.
Assume that $f\in H^1_{loc}$ and $\varphi : \R\to\R$ is an orientation preserving  $C^1$-diffeomorphism of $\R$.
Let $\psi\in C^\infty_0$ be a test function such that $\supp\psi\subseteq (-R,R)$, $0<R<\infty$.
As $f\in H^1_{loc}$ there exists $(f_k)_{k\ge 1}$, $f_k\in C^\infty\big([\varphi(-R),\varphi(R)]\big)$ such that
$f_k\to f$ $(k\to\infty)$ in $H^1\big((\varphi(-R),\varphi(R))\big)$. Then,
\begin{eqnarray}
\langle (f\circ\varphi)',\psi\rangle&:=&-\int_{-\infty}^\infty f(\varphi(x))\cdot \psi'(x)\,dx=
-\int_{\varphi(-R)}^{\varphi(R)} f(y)\cdot\frac{\psi'(\varphi^{-1}(y))}{\varphi'(\varphi^{-1}(y))}\,dy\nonumber\\
&=&-\lim_{k\to\infty}
\int_{\varphi(-R)}^{\varphi(R)} f_k(y)\cdot\frac{\psi'(\varphi^{-1}(y))}{\varphi'(\varphi^{-1}(y))}\,dy\label{eq:weak1}
\end{eqnarray}
and
\begin{eqnarray}
\lim_{k\to\infty}\int_{-R}^R(f_k(\varphi(x)))'\cdot \psi(x)\,dx
&=&\lim_{k\to\infty}\int_{-R}^R f_k'(\varphi(x))\cdot\varphi'(x)\cdot \psi(x)\,dx\nonumber\\
&=&\!\!\!\lim_{k\to\infty}\int_{\varphi(-R)}^{\varphi(R)}\!\!\!\!\!\!\!
f_k'(y)\cdot\varphi'(\varphi^{-1}(y))\cdot\frac{\psi(\varphi^{-1}(y))}{\varphi'(\varphi^{-1}(y))}\,dy\nonumber\\
&=&\int_{\varphi(-R)}^{\varphi(R)}
f'(y)\cdot\varphi'(\varphi^{-1}(y))\cdot\frac{\psi(\varphi^{-1}(y))}{\varphi'(\varphi^{-1}(y))}\,dy\nonumber\\
&=&\int_{-\infty}^{\infty} f'(\varphi(x))\cdot\varphi'(x)\cdot\psi(x)\,dx\nonumber\\
&=&\langle f'\circ\varphi\cdot\varphi',\psi\rangle\,.\label{eq:weak2}
\end{eqnarray}
Combining \eqref{eq:weak1} with \eqref{eq:weak2} we see that
\begin{equation}\label{eq:chain_rule}
(f\circ\varphi)'= f'\circ\varphi\cdot\varphi'\,.
\end{equation}
Using a change of variables one easily sees that $f\circ\varphi$ and $f'\circ\varphi$ are
locally square integrable.
This and \eqref{eq:chain_rule} then imply that $f\circ\varphi\in H^1_{loc}$.
The proof of the last statement of $(i)$ follows from \eqref{eq:chain_rule} in a similar way.
\end{proof}

\begin{Lem}\label{lem:decay}
Let $m\ge 1$ and $N\in\R$. Then there exists $C>0$ such that for any $f\in W^m_N$ and
for any $0\le j\le m-1$,
\[
\sup_{x\in\R}|f^{(j)}(x)|\x^{N+j+\frac{1}{2}}\le C\|f\|_{W_N^m}
\]
and
\[
\lim_{|x|\to\infty}|f^{(j)}(x)|\x^{N+j+\frac{1}{2}}=0\,.
\]
If $N\ge 0$ one has the continuous inclusions $W_N^m\subseteq H^m\subseteq C^{m-1}$.
\end{Lem}
\begin{proof}
First assume that $f\in W^1_N$ and $N\in\R$. Consider the function
\[
g(x):=f(x)^2\x^{2N+1} \,.
\]
A direct differentiation shows that
\begin{equation}\label{eq:g'-estimate}
|g'|\le (2|N|+1)\big(\x^N|f|\big)^2+2\big(\x^N|f|\big)\big(\x^{N+1}|f'|\big)\in L^1\,.
\end{equation}
As $g'\in L^1$, $g(x)=c_-+\int_{-\infty}^x g'(y)\,dy=c_+-\int_x^\infty g'(y)\,dy$ where
$c_\pm\ge 0$ are non-negative constants. In particular,
\[
\lim_{x\to\pm\infty}g(x)=c_\pm\ge 0\,.
\] 
Now, assume that $c_+>0$. Then, there exist $\alpha\ge 0$ and $\varepsilon>0$ such that
$g(x)\ge\varepsilon^2>0$ for any $x\ge\alpha$. In particular, for any $x\ge\alpha$,
\[
|f(x)|\x^N>\frac{\varepsilon}{\x^{1/2}},
\]
that contradicts $\x^N f\in L^2$. Hence, $c_+=0$. Arguing in a similar way we see that $c_-=0$.
In particular,
\begin{equation}\label{eq:g_representation}
g(x)=\int_{-\infty}^x g'(y)\,dy=-\int_x^\infty g'(y)\,dy
\end{equation}
and
\[
\lim_{|x|\to\infty} g(x)=0\,.
\]
This shows that 
\begin{equation}\label{eq:weight_small_o}
\lim_{|x|\to\infty} |f(x)|\x^{N+\frac{1}{2}}=0\,.
\end{equation}
It follows from \eqref{eq:g'-estimate} and \eqref{eq:g_representation} that
\[
|g(x)|\le\int_{-\infty}^\infty|g'(y)|\,dy\le\big(2|N|+3\big)\|f\|_{W^1_N}\,.
\]
Hence, there exists $C_N>0$ such that for any $f\in W^1_N$ and for any $x\in\R$,
\begin{equation}\label{eq:weight1}
\x^{N+\frac{1}{2}}|f(x)|\le C_N\|f\|_{W^1_N}\,.
\end{equation}
Now, assume that $f\in W^m_N$, $m\ge 1$, and $N\in\R$.
Take $0\le j\le m-1$. Then 
\[
f^{(j)}\in W^{m-j}_{N+j}\subseteq W^1_{N+j}.
\]
In view of \eqref{eq:weight1} we have that for any $x\in\R$,
\[
\x^{N+j+\frac{1}{2}}|f^{(j)}|\le C_{N+j}\|f^{(j)}\|_{W^1_{N+j}}\le C_{N+j}\|f\|_{W^m_N}\,.
\]
This prove the lemma with $C:=\max\limits_{0\le j\le m-1} C_{N+j}$.
In view of \eqref{eq:weight_small_o} we also get that 
$\lim\limits_{|x|\to\infty} |f^{(j)}(x)|\x^{N+j+\frac{1}{2}}=0$.
The last statement of the lemma follows from Sobolev's embedding theorem.
\end{proof}
Arguing in a similar way one proves.
\begin{Lem}\label{lem:decay'}
Let $m\ge 1$ and $N\in\R$. Then there exists $C>0$ such that for any $f\in H^m_N$ and for any 
$0\le j\le m-1$,
\[
\sup_{x\in\R}|f^{(j)}(x)|\x^N\le C\|f\|_{H_N^m}
\]
and
\[
\lim_{|x|\to\infty}|f^{(j)}(x)|\x^N=0\,.
\]
If $N\ge 0$ one has the continuous inclusions $H_N^m\subseteq H^m\subseteq C^{m-1}$.
\end{Lem}
\proof
The proof of this statement follows the arguments of the proof of Lemma \ref{lem:decay} applied to
the function $g(x):=f(x)^2\x^{2N}$.
\endproof

The following Lemma can be proved in a straightforward way.
\begin{Lem}\label{lem:W-properties}
Assume that $m_1, m_2\ge 1$, $N_1, N_2\ge 0$, and $k\ge 0$.
Then the mappings,
\[
\begin{array}{l}
(f,g)\mapsto f\cdot g,\;\;\;W_{N_1}^{m_1}\times W_{N_2}^{m_2}\to W_{N_1+N_2}^{\min(m_1,m_2)},\\
f\mapsto f\cdot\frac{1}{\x^k},\;\;\;W_{N_1}^{m_1}\to W_{N_1+k}^{m_1},\\
f\mapsto f\cdot\frac{x}{\x^{k+1}},\;\;\;W_{N_1}^{m_1}\to W_{N_1+k}^{m_1},
\end{array}
\]
and for $m_1\ge 2$,
\begin{equation*}
f\mapsto f',\;\;\;W_{N_1}^{m_1}\to W_{N_1+1}^{m_1-1},
\end{equation*}
are continuous. 
In particular, for any $m\ge 1$, $N\ge 0$, $W_N^m$ is a ring with continuous (pointwise) product.
\end{Lem}

As a consequence from Lemma \ref{lem:W-properties} we get
\begin{Lem}\label{lem:A-properties}
Assume that $m_1, m_2\ge 1$, $N_1, N_2, n_1, n_2, k\ge 0$.
Then the mappings,
\begin{equation*}
\begin{array}{l}
(f,g)\mapsto f\cdot g,\,\,\,
\A_{n_1,N_1}^{m_1}\times \A_{n_2,N_2}^{m_2}\to\A_{n_1+n_2,\min(N_1+n_2,N_2+n_1)}^{\min(m_1,m_2)},\\
f\mapsto f\cdot\frac{1}{\x^k},\,\,\,\A_{n_1,N_1}^{m_1}\to\A_{n_1+k,N_1+k}^{m_1},\\
f\mapsto f\cdot\frac{x}{\x^{k+1}},\,\,\,\A_{n_1,N_1}^{m_1}\to\A_{n_1+k,N_1+k}^{m_1},
\end{array}
\end{equation*}
and for $m_1\ge 2$,
\begin{equation*}
f\mapsto f',\;\;\;\A_{n_1,N_1}^{m_1}\to\A_{n_1+1,N_1+1}^{m_1-1},
\end{equation*}
are continuous.
In particular, for any $m\ge 1$, $N\ge 0$, and $n\ge 0$, $\A_{n,N}^m$ is a ring with continuous
(pointwise) product.
\end{Lem}

\begin{Rem}\label{rem:smoothness}
As the mappings considered in Lemma \ref{lem:W-properties} and Lemma \ref{lem:A-properties} are
multilinear and continuous they are $C^\infty$-smooth.
\end{Rem}

Similar arguments imply the following analogs of 
Lemma \ref{lem:W-properties} and Lemma \ref{lem:A-properties}.
\begin{Lem}\label{lem:H-properties}
Assume that $m_1, m_2\ge 1$, $N_1, N_2\ge 0$, and $k\ge 0$.
Then the mappings,
\[
\begin{array}{l}
(f,g)\mapsto f\cdot g,\;\;\;H_{N_1}^{m_1}\times H_{N_2}^{m_2}\to H_{N_1+N_2}^{\min(m_1,m_2)},\\
f\mapsto f\cdot\frac{1}{\x^k},\;\;\;H_{N_1}^{m_1}\to H_{N_1+k}^{m_1},\\
f\mapsto f\cdot\frac{x}{\x^{k+1}},\;\;\;H_{N_1}^{m_1}\to H_{N_1+k}^{m_1},
\end{array}
\]
and for $m_1\ge 2$,
\begin{equation*}
f\mapsto f',\;\;\;H_{N_1}^{m_1}\to H_{N_1}^{m_1-1},
\end{equation*}
are continuous. 
In particular, for any $m\ge 1$, $N\ge 0$, $H_N^m$ is a ring with continuous (pointwise) product.
\end{Lem}

\begin{Lem}\label{lem:AH-properties}
Assume that $m_1, m_2\ge 1$, $N_1, N_2, n_1, n_2, k\ge 0$.
Then the mappings,
\begin{equation*}
\begin{array}{l}
(f,g)\mapsto f\cdot g,\;\;\;
\AH_{n_1,N_1}^{m_1}\times \AH_{n_2,N_2}^{m_2}\to\AH_{n_1+n_2,\min(N_1+n_2,N_2+n_1)}^{\min(m_1,m_2)},\\
f\mapsto f\cdot\frac{1}{\x^k},\;\;\;\AH_{n_1,N_1}^{m_1}\to\AH_{n_1+k,N_1+k}^{m_1},\\
f\mapsto f\cdot\frac{x}{\x^{k+1}},\;\;\;\AH_{n_1,N_1}^{m_1}\to\AH_{n_1+k,N_1+k}^{m_1},
\end{array}
\end{equation*}
and for $m_1\ge 2$,
\begin{equation*}
f\mapsto f',\;\;\;\AH_{n_1,N_1}^{m_1}\to\AH_{n_1+1,N_1}^{m_1-1},
\end{equation*}
are continuous.
In particular, for any $m\ge 1$, $N\ge 0$, and $n\ge 0$, $\AH_{n,N}^m$ is a ring with continuous
(pointwise) product.
\end{Lem}

\noindent{\em Left composition with analytic maps:}
First we introduce some additional notation. Take $r\ge 0$ and denote 
\[
\Rr:=\{x\in\R\,|\,|x|\ge r\}\,.
\]
For any $m\ge 1$, $N\ge 0$, and $n\ge 1$, consider the weighted Sobolev space
\[
W^m_N(\Rr):=\big\{f\in H^m_{loc}(\Rr)\,\big|\,\x^N f,...,\x^{N+m} f^{(m)}\in L^2(\Rr)\big\}
\]
supplied with the norm
\[
\|f\|_{W^m_N(\Rr)}:=\Big(\sum_{j=0}^m\int_{\Rr}|\x^{N+j} f^{(j)}(x)|^2\,dx\Big)^{1/2}
\]
as well as the modified asymptotic space 
\[
\A^m_{n,N}(\Rr):=
\Big\{u=\sum_{k=n}^N\Big(a_k\frac{1}{\x^k}+b_k\frac{x}{\x^{k+1}}\Big)+f\,\Big|\,f\in W^m_N(\Rr)\Big\}
\]
supplied with the {\em weighted} norm
\begin{equation}\label{eq:modified_norm}
\|u\|_{\A^m_{n,N}(\Rr)}:=\sum_{k=n}^N\big(|a_k|+|b_k|\big)/\rr^{k}+\|f\|_{W^m_N(\Rr)}\,.
\end{equation}
Note that if $r=0$ then $\A^m_{n,N}(\R_r)\equiv\A^m_{n,N}$ and the corresponding norms coincide.
It is also clear that for any given $r\ge 0$ the restriction map,
\[
\A^m_{n,N}\to\A^m_{n,N}(\Rr),\,\,\,u\mapsto u|_{\Rr},
\]
is continuous and\footnote{For simplicity of notation when estimating the $\A^m_{n,N}(\Rr)$-norm of
$u|_{\A^m_{n,N}(\Rr)}$ we write $u$ instead of $u|_{\A^m_{n,N}(\Rr)}$.}
\[
\|u\|_{\A^m_{n,N}(\Rr)}\to 0,\,\,r\to\infty\,.
\]
\begin{Lem}\label{lem:r-spaces} Assume that $m\ge 1$, $N\ge 0$, and $n\ge 1$. 
Then the following two statements hold

\noindent $1)$ $u\in\A^m_{n,N}$ if and only if $u\in H^m_{loc}\big((-(r+1),r+1)\big)$ and 
$u|_{\Rr}\in\A^m_{n,N}(\Rr)$.

\noindent $2)$ There exists a positive constant $C>0$ so that for any $r\ge 0$, for any integer $p\ge 0$, 
and for any $u\in\A^m_{n,N}(\Rr)$,
\begin{equation}\label{eq:product_powers}
\|u^p u^{N+1}\|_{\A^m_{n,N}(\Rr)}\le \big(C\|u\|_{\A^m_{n,N}(\Rr)}\big)^p\|u^{N+1}\|_{\A^m_{n,N}(\Rr)}\,.
\end{equation}
\end{Lem}
\begin{proof}
The proof of $1)$ follows directly from the definition of the spaces involved.
Let us prove $2)$. Take $u\in\A^m_{n,N}(\Rr)$ and $g\in W^m_N(\Rr)\subseteq\A^m_{n,N}$. 
Then $u=\sum_{k=n}^N(a_k\frac{1}{\x^k}+b_k\frac{x}{\x^{k+1}})+f$, $f\in W^m_N(\Rr)$.
For any $n\le k\le N$ the pointwise product $(a_k/\x^k)\cdot g\in W^m_{N+k}\subseteq W^m_N$
(Lemma \ref{lem:W-properties}) and
\begin{eqnarray}
\|(a_k/\x^k)\cdot g\|_{A^m_{n,N}(\Rr)}^2&=&
\sum_{j=0}^m\int_{\Rr}\Big|\Big(\frac{a_k}{\x^k}\,g(x)\Big)^{(j)}\x^{N+j}\Big|^2\,dx\nonumber\\
&\le&C_1^2\,|a_k|^2\sum_{j=0}^m\int_{\Rr}
\Big(\sum_{l=0}^j\frac{|g^{(l)}(x)|}{\x^{k}}\,\x^{N+l}\Big)^2\,dx\nonumber\\
&\le&C_1^2\,\frac{|a_k|^2}{\rr^{2k}}\sum_{j=0}^m\int_{\Rr}
\Big(\sum_{l=0}^j|g^{(l)}(x)|\,\x^{N+l}\Big)^2\,dx\nonumber\\
&\le&C_2^2\,\|(a_k/\x^k)\|_{A^m_{n,N}(\Rr)}^2\|g\|_{A^m_{n,N}(\Rr)}^2\label{eq:product1}
\end{eqnarray}
where the positive constants $C_1$ and $C_2$ are independent of the choice of $a_k\in\R$ and $g\in W^m_N(\Rr)$.
Similar arguments show that there exists $C_3>0$ such that for any $b_k\in\R$ and $g\in W^m_N(\Rr)$,
\begin{equation}\label{eq:product2}
\|(b_k x/\x^{k+1})\cdot g\|_{A^m_{n,N}(\Rr)}\le 
C_3\|(b_k x/\x^{k+1})\|_{A^m_{n,N}(\Rr)}\|g\|_{A^m_{n,N}(\Rr)}\,.
\end{equation}
By the product rule one also sees that there exists $C_4>0$ such that for any $f,g\in W^m_N$,
\begin{equation}\label{eq:product3}
\|f g\|_{\A^m_{n,N}(\Rr)}\le C_4\|f\|_{\A^m_{n,N}(\Rr)}\|g\|_{\A^m_{n,N}(\Rr)}\,.
\end{equation}
Combining \eqref{eq:product1}, \eqref{eq:product2}, and \eqref{eq:product3}, we get that there exists
a positive constant $C>0$ such that for any $u\in\A^m_{n,N}(\Rr)$ and $g\in W^m_N(\Rr)$,
\begin{equation}\label{eq:inequality4}
\|u g\|_{\A^m_{n,N}(\Rr)}\le C\|u\|_{A^m_{n,N}(\Rr)}\|g\|_{A^m_{n,N}(\Rr)}\,.
\end{equation}
As $n\ge 1$, the product $u^{N+1}\in W^m_N(\Rr)$ (Lemma \ref{lem:W-properties}) and the claimed
inequality 
\eqref{eq:product_powers} then follows from \eqref{eq:inequality4} and an induction argument.
\end{proof}
As a corollary we obtain the following
\begin{Prop}\label{prop:analyticity_general}
Assume that $m\ge 1$, $N\ge 0$, and $n\ge 1$.
Let $\alpha : I\to\R$ be a real-analytic function in the open interval $I\subseteq\R$ that 
contains zero, $0\in I$. Then, if $u\in\A^m_{n,N}$ and $\mathop{\rm Image}(u)\subseteq I$ then
$\alpha\circ u-\alpha(0)\in\A^m_{n,N}$.\footnote{$\mathop{\rm Image}(u):=\{u(x)\,|\,x\in\R\}$.} 
The same statement holds if $\A$ is replaced by $\AH$.
\end{Prop}
\begin{proof}
Take $u\in\A^m_{n,N}$ so that $\mathop{\rm Image}(u)\subseteq I$.
As $0\in I$ and as $\alpha : I\to\R$ is real-analytic there exists $\delta>0$ such that
\[
\alpha(x)=\sum_{j=0}^\infty\alpha_j x^j
\]
where the series converges absolutely on $(-\delta,\delta)$.
As $n\ge 1$ we see from Lemma \ref{lem:decay} that there exists $r\ge 0$ so that
for any $x\in\Rr$,
\[
|u(x)|<\delta\,.
\]
In particular, for any $x\in\Rr$,
\[
(\alpha\circ u)(x)=\sum_{j=0}^\infty\alpha_j u(x)^j,
\]
where the series converges absolutely. For any $x\in\Rr$, we have
\begin{equation}\label{eq:splitted_sum}
(\alpha\circ u)(x)-\alpha_0=\sum_{j=1}^N\alpha_j u(x)^j+
\sum_{p=0}^\infty\alpha_{N+1+p}u(x)^p u(x)^{N+1}\,.
\end{equation}
In view of Lemma \ref{lem:r-spaces},
\begin{equation}
\big\|\sum_{p=0}^\infty\alpha_{N+1+p}u^p u^{N+1}\big\|_{A^m_{n,N}(\Rr)}\!\le\!
\sum_{p=0}^\infty|\alpha_{N+1+p}|\big(C\|u\|_{\A^m_{n,N}(\Rr)}\big)^p\|u^{N+1}\|_{\A^m_{n,N}(\Rr)}\,.
\end{equation}
By taking $r\ge 0$ so that
\[
\|u\|_{\A^m_{n,N}(\Rr)}\le\delta/C
\]
we obtain that \eqref{eq:splitted_sum} converges in $\A^m_{n,N}(\Rr)$.
Hence,
\[
\big(\alpha\circ u-\alpha(0)\big)|_{\Rr}\in\A^m_{n,N}(\Rr)\,.
\]
As $u\in H^m_{loc}\big(-(r+1),r+1\big)$, $\mathop{\rm Image}(u)\subseteq I$, and as 
$\alpha$ is real-analytic in $I$, one concludes by standard arguments that 
\[
\alpha\circ u\in H^m_{loc}\big(-(r+1),r+1\big)\,.
\]
The statement of the Proposition now follows from Lemma \ref{lem:r-spaces}, $1)$.
The case of the space $\AH^m_{n,N}$ is treated in the same way.
\end{proof}

We will use Proposition \ref{prop:analyticity_general} for proving the following important Lemmas. 
\begin{Lem}\label{lem:1/phi'}  Assume that $\psi:=1+f>0$ where $f\in\A_{n,N}^m$, $m\ge 1$, $N\ge 0$, and
$n\ge 1$. Then $\frac{1}{\psi}-1\in\A_{n,N}^m$ and there exists an open neighborhood $\U$ of zero in
$\A_{n,N}^m$ such that for any $g\in\U$, $\psi+g>0$, and the mapping,
\[
g\mapsto\frac{1}{\psi+g}-1,\;\;\;\U\to \A_{n,N}^m,
\]
is real-analytic. The same statements are true if $\A$ is replaced by $\AH$.
\end{Lem}
\begin{proof}
The fact that $\frac{1}{\psi}-1\in\A^m_{n,N}$ follows from Proposition \ref{prop:analyticity_general}
applied
to the real-analytic function $\alpha(x):=1/(1+x)$, $\alpha : (-1,\infty)\to\R$.
Let us prove the second statement of the Lemma.
It follows from the continuity of the pointwise product in $\A^m_{n,N}$ and the continuity of
the inclusion $\A^m_{n,N}\subseteq L^\infty$ that there exists an open neighborhood $\U$ of zero in 
$\A^m_{n,N}$ such that 
\[
\Big\|g\cdot\frac{1}{\psi}\Big\|_{\A^m_{n,N}}<1\,\,\,\,\mbox{and}\,\,\,
\Big\|g\cdot\frac{1}{\psi}\Big\|_{L^\infty}<1\,.
\]
Then for any $g\in\U$,
\[
\frac{1}{\psi+g}=\frac{1}{\psi}\cdot\frac{1}{1+g\cdot\frac{1}{\psi}}=
\frac{1}{\psi}\left(1+\sum_{j=1}^\infty(-1)^j\Big(g\cdot\frac{1}{\psi}\Big)^j\right)
\]
where the series converges in $\A^m_{n,N}$.
The case of the space $\AH^m_{n,N}$ is treated in the same way. 
\end{proof}

Similar arguments show that the following Lemma hold.
\begin{Lem}\label{lem:1/x-composition}
Assume that $m\ge 2$, $N\ge 0$, and $n\ge 0$. Then for any $\varphi\in\A\D^m_{n,N}$
\[
\frac{1}{\langle\cdot\rangle}\circ\varphi=\frac{1}{\x}\big(1+w\big),\,\,\,w\in\A^m_{n+1,N+1}
\]
where $\frac{1}{\langle\cdot\rangle}\circ\varphi$ stands for $1/\sqrt{1+\varphi(x)^2}$.
Moreover, the mapping,
\[
\varphi\mapsto\frac{1}{\langle\cdot\rangle}\circ\varphi-\frac{1}{\x},\,\,\,\A\D^{m}_{n,N}\to\A^m_{n+2,N+2}
\]
is real-analytic. The same statements are true if $\A$ is replaced by $\AH$.
\end{Lem}
\proof 
Take $\varphi\in\A\D^m_{n,N}$. Then $\varphi(x)=x+u(x)$, $u\in\A^m_{n,N}$, and
\begin{equation}\label{eq:useful_form}
\frac{1}{\langle\cdot\rangle}\circ\varphi=\frac{1}{\sqrt{1+(x+u)^2}}=
\frac{1}{\x}\cdot\frac{1}{\sqrt{1+\Big(2\frac{x}{\x}\frac{u}{\x}+\big(\frac{u}{\x}\big)^2\Big)}}\,.
\end{equation}
In view of Lemma \ref{lem:A-properties}, the expression 
$2\frac{x}{\x}\frac{u}{\x}+\big(\frac{u}{\x}\big)^2\in\A^m_{n+1,N+1}$.
As $1+\Big(2\frac{x}{\x}\frac{u}{\x}+\big(\frac{u}{\x}\big)^2\Big)>0$ we obtain from 
Proposition \ref{prop:analyticity_general}, applied to the real-analytic function $\alpha(x):=1/\sqrt{1+x}$,
$\alpha : (-1,\infty)\to\R$, that 
\[
\frac{1}{\sqrt{1+\Big(2\frac{x}{\x}\frac{u}{\x}+\big(\frac{u}{\x}\big)^2\Big)}}-1\in\A^m_{n+1,N+1}\,.
\]
This together with \eqref{eq:useful_form} completes the proof of the first statement of the Lemma.
The second statement follows as in the proof of the second statement of Lemma \ref{lem:1/phi'}.
The case of the space $\AH^m_{n,N}$ is treated in the same way.
\endproof



\begin{thebibliography}{99}   

\bibitem{Arn} V. Arnold, {\em Sur la geometri{\'e} differentielle des groupes de Lie
de dimension infinie et ses applications {\`a} l'hydrodynamique des fluids parfaits},
Ann. Inst. Fourier, $\bf 16$, 1(1966), 319-361


\bibitem{BS1} I. Bondareva, M. Shubin, {\em Growing asymptotic solutions of
the Korteweg-de Vries equation and of its higher analogues}, Dokl. Akad. Nauk SSSR,
$\bf 267$(1982), no. 5, 1035-1038

\bibitem{BS2} I. Bondareva, M. Shubin, {\em Uniqueness of the solution of the Cauchy problem
for the Korteweg-de Vries equation in classes of increasing functions},
Vestnik Moskov. Univ. Ser. I Mat. Mekh, 1985, no. 3, 35-38

\bibitem{BS3} I. Bondareva, M. Shubin, {\em Equations of Korteweg-de Vries type in classes of
increasing functions}, J. Soviet Math., $\bf 51$(1990), no. 3, 2323-2332


\bibitem{CH} R. Camassa, D. Holm, {\em An integrable shallow water equation with
peaked solitons}, Phys. Rev. Lett, $\bf 71$(1993), 1661-1664

\bibitem{Con2} A. Constantin, {\em Existence of permanent and breaking waves for a
shallow water equation: a geometric approach}, Ann. Inst. Fourier, Grenoble, $\bf 50$(2000), no. 2,
321-362

\bibitem{CE2}  A. Constantin, J. Escher, {\em Global existence and blow-up for a shallow water
equation}, Annali Sc. Norm. Sup. Pisa, $\bf 26$(1998), 303-328

\bibitem{CE3} A. Constantin, J. Escher, {\em On the blow-up rate and
the blow-up set of breaking waves for a shallow water equation}, Math. Z., 
$233$(2000), 75-91

\bibitem{CE4} A. Constantin, J. Escher,
{\em Global weak solutions for a shallow water equation},
Indiana Univ. Math. J., $\bf 47$(1998), no. 4, 1527-1545


\bibitem{DKT} C. De Lellis, T. Kappeler, P. Topalov, {\em Low regularity solutions of
the Camassa-Holm equation}, Comm. Partial Differential Equations, $\bf 32$(2007), no. 1-3, 87-126

\bibitem{EM} D. Ebin, J. Marsden, {\em Groups of diffeomorphisms and the motion
of an incompressible fluid}, Ann. Math., $\bf 92$(1970), 102-163

\bibitem{FF} A. Fokas, B. Fuchssteiner, {\em Symplectic structures,
their B{\"a}cklund transformation and hereditary symmetries},
Physica D, $\bf 4$(1981), 47-66 


\bibitem{HS1} D. Holm, M. Staley, {\em Wave structure and nonlinear balances in a family of evolutionary
PDEs}, SIAM J. Applied Dynamical Systems, $\bf 2$(2003), no. 3, 323-380

\bibitem{HS2} D. Holm, M. Staley, {\em Nonlinear balance and exchange of stability in dynamics of solitons,
peakons, ramps/cliffs and leftons in $1+1$ nonlinear PDE}, Phys. Lett. A, $\bf 308$(2003), 437-444

\bibitem{IKT} H. Inci, T. Kappeler, P. Topalov, {\em On the regularity of the composition of
diffeomorphisms}, Mem. Amer. Math. Soc., $\bf 226$(2013), no. 1062

\bibitem{KPST} T. Kappeler, P. Perry, M. Shubin, P. Topalov, {\em Solutions of mKdV in classes of 
functions unbounded at infinity}, J. Geom. Anal., $\bf 18$(2008), no. 2, 443-477



\bibitem{Lang} S. Lang, {\em Differential manifolds},
Addison-Wesley Series in Mathematics, 1972

\bibitem{McOwenTopalov1} R. McOwen, P. Topalov, {\em Groups of asymptotic diffeomorphisms}, preprint


\bibitem{Mis1} G. Misiolek, {\em A shallow water equation as a geodesic flow
on the Bott-Virasoro group}, J. Geom. Phys., $\bf 24$(1998), 203-208

\bibitem{Mis2} G. Misiolek, {\em Classical solutions of the periodic
Camassa-Holm equation}, GAFA, $\bf 12$(2002), 1080-1104

\bibitem{OK} V. Ovsienko, B. Khesin, {\em Korteweg-de Vries superequations
as an Euler equation}, Functional Anal. Appl., $\bf 21$(1987), 81-82

\end{thebibliography}
\end{document}